\documentclass[11pt]{amsart}

\usepackage{stmaryrd}
\usepackage{amscd,amsxtra,amssymb,mathrsfs, bbm, bigints}

\usepackage{multirow}
\usepackage[pdftex]{hyperref}
\usepackage{longtable}
\usepackage{textgreek}
\usepackage{tikz}
\usepackage[utf8]{inputenc}
\usepackage{float}
\usepackage{color}
\usetikzlibrary{calc, positioning, shapes, fit, matrix, decorations}
\usetikzlibrary{decorations.shapes}

\usepackage[cmtip]{xypic}
\usepackage[all]{xy}
\xyoption{arc}

\newtheorem{theorem}{Theorem}[section]
\newtheorem{corollary}[theorem]{Corollary}
\newtheorem{lemma}[theorem]{Lemma}
\newtheorem{proposition}[theorem]{Proposition}

\theoremstyle{definition}
\newtheorem{definition}[theorem]{Definition}
\newtheorem{remark}[theorem]{Remark}
\newtheorem{example}[theorem]{Example}

\numberwithin{equation}{section}

\DeclareMathAlphabet{\mathpzc}{OT1}{pzc}{m}{it}

\renewcommand{\dim}{\mathsf{dim}}

\DeclareMathOperator{\Coh}{\mathsf{Coh}}

\DeclareMathOperator{\Pic}{\mathsf{Pic}}

\DeclareMathOperator{\Hom}{\mathsf{Hom}}

\DeclareMathOperator{\Ext}{\mathsf{Ext}}

\DeclareMathOperator{\GL}{\mathsf{GL}}

\DeclareMathOperator{\End}{\mathsf{End}}
\DeclareMathOperator{\Mat}{\mathsf{Mat}}

\input xy
\xyoption{all}

\def\bu{{\scriptscriptsize\bullet}}

\newcommand{\FF}{\mathbb{F}}

\renewcommand{\mod}{\mathsf{mod}}
\def\bu{{\scriptscriptstyle\bullet}}

\newcommand{\kA}{\mathcal{A}}
\newcommand{\kB}{\mathcal{B}}
\newcommand{\kE}{\mathcal{E}}
\newcommand{\kF}{\mathcal{F}}

\newcommand{\kO}{\mathcal{O}}
\newcommand{\kL}{\mathcal{L}}
\newcommand{\kP}{\mathcal{P}}

\newcommand{\kV}{\mathcal{V}}
\newcommand{\kW}{\mathcal{W}}
\newcommand{\kT}{\mathcal{T}}

\newcommand{\lar}{\longrightarrow}

\newcommand{\RR}{\mathbb R}
\newcommand{\CC}{\mathbb C}

\newcommand{\ZZ}{\mathbb{Z}}

\newcommand{\nn}{\mathsf{n}}

			\def\bu{\bullet}

\def\Mat{\mathop\mathrm{Mat}}

\def\8{\infty}			
	\def\+{\oplus}		
\def\*{\otimes}

\def\kA{\mathcal A} 
\def\kB{\mathcal B} \def\kO{\mathcal O}
\def\kC{\mathcal C} \def\kP{\mathcal P}
 \def\kQ{\mathcal Q}
\def\kE{\mathcal E} \def\kR{\mathcal R}
\def\kF{\mathcal F} \def\kS{\mathcal S}
 \def\kT{\mathcal T}
 \def\kU{\mathcal U}
 \def\kV{\mathcal V}
 \def\kW{\mathcal W}
 
\def\kL{\mathcal L}

	\def\NN{\mathbb N}

\newcommand{\rightarrowdbl}{\longrightarrow\mathrel{\mkern-14mu}\rightarrow}

\def\DMO{\DeclareMathOperator}
\DMO{\ob}{Ob}            \DMO{\mor}{Mor}
\DMO{\Ker}{Ker}
\DMO{\id}{Id}

\title[Algebraic geometry of the multilayer model on a torus]{Algebraic geometry of the multilayer model of the fractional quantum Hall effect on a torus}

\author{Igor Burban}
\address{
Universit\"at Paderborn,
Institut f\"ur Mathematik,
Warburger Strasse 100,
33098 Paderborn,
Germany
}
\email{burban@math.uni-paderborn.de}

\author{Semyon Klevtsov}
\address{
IRMA, Universit\'e de Strasbourg\\
UMR 7501, 7 rue Ren\'e  Descartes,\\
67084 Strasbourg, \\
France
}
\email{klevtsov@unistra.fr}

\begin{document}

\begin{abstract}
In 1993 Keski-Vakkuri and Wen introduced a model for the fractional quantum Hall effect based on  multilayer two-dimensional electron systems satisfying  quasi-periodic boundary conditions. Such a model is essentially specified by a choice of a complex torus $E$ and a symmetric positively definite matrix $K$ of size $g$ with non-negative integral coefficients, satisfying some further constraints.

The space of the corresponding wave functions turns out to be $\delta$-dimensional, where $\delta$ is the determinant of $K$.  We construct   a hermitian holomorphic bundle of rank $\delta$ on the abelian variety $A$ (which is the $g$-fold product of the torus $E$ with itself), whose fibres can be identified with the space of wave function of Keski-Vakkuri and Wen. A rigorous construction of this ``magnetic bundle'' involves the technique of Fourier-Mukai transforms on abelian varieties. The constructed  bundle turns out to be  simple and semi-homogeneous and it  can be equipped with two different  (and natural) hermitian metrics: the one coming from the center-of-mass dynamics and the one coming from the Hilbert space of the underlying many-body system. We prove that the canonical Bott--Chern connection of the first hermitian metric is always projectively flat and give sufficient conditions for this property for the second hermitian metric.
\end{abstract}

\maketitle

\section{Introduction} The experimental discovery of the integer \cite{Klitzing} and then the fractional quantum Hall effects \cite{Tsui}  (IQHE, FQHE)  are widely considered to be the major events in the condensed matter physics in the second half of the twentieth century. Both experimental and theoretical aspects of these phenomena for other several decades 
continued to attract enormous attention and, in particular, have led to several Nobel prizes -- awarded to von Klitzing (1985), Laughlin,  St\"ormer and Tsui (1998), and to Haldane and Thouless (2016).

The most basic experimental setting involves certain two-dimensional electronic systems such as, originally, gallium arsenide heterostructures 
and, more recently,  graphene, and can be described as follows. The system is subjected to a perpendicular magnetic field $B$ which deflects the longitudinal electric current $I$ by means of the Lorentz force. As a consequence, a transversal tension $V$ proportional to $I$, $V=I/\sigma$ builds up, with the coefficient $\sigma$ called the Hall conductance.

\begin{figure}[ht!]
\includegraphics[width=7cm]{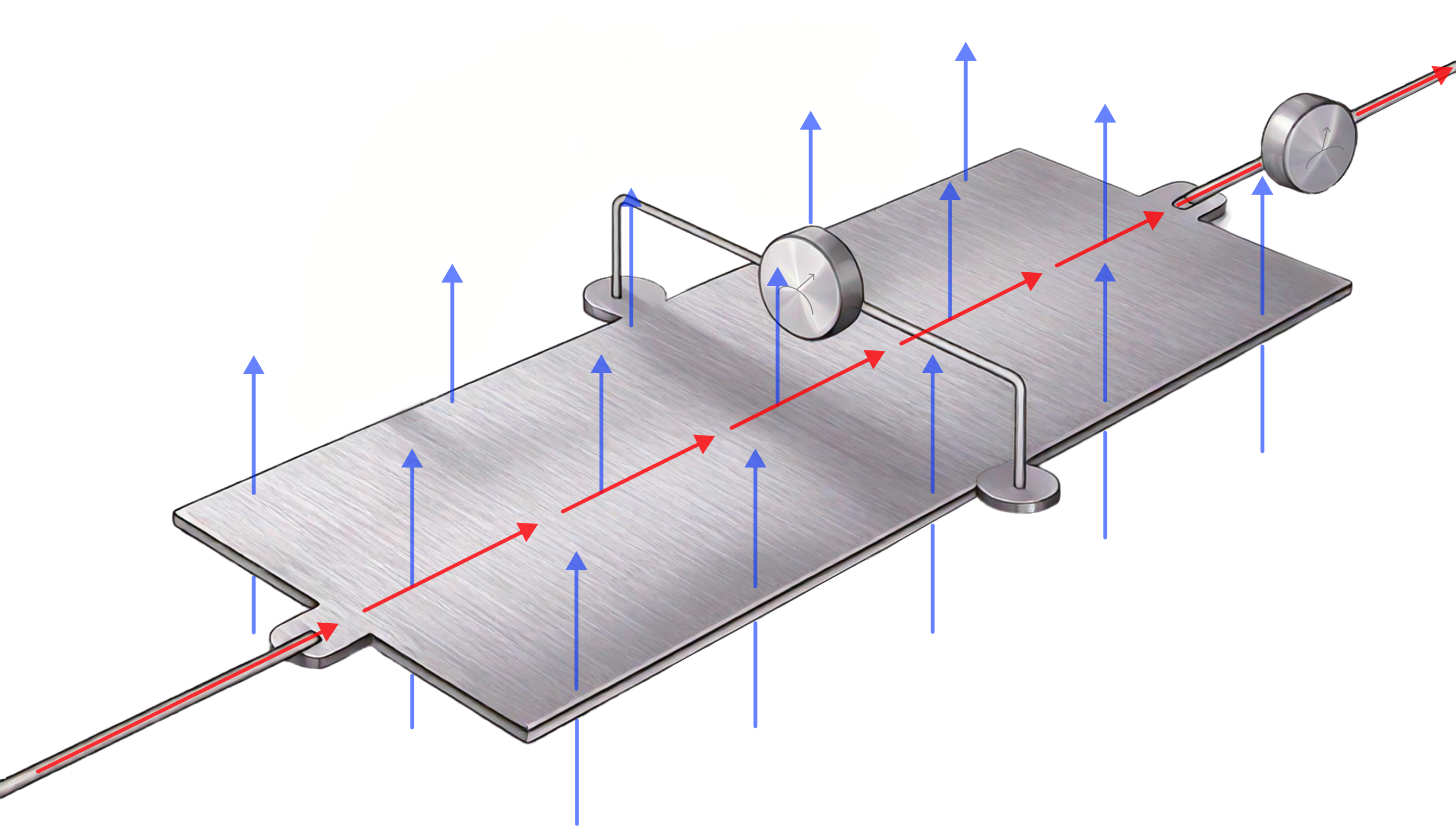}
\end{figure}

In classical electrodynamics, $\sigma$ is inverse proportional to the magnetic field $B$.
When subjected to large magnetic field and low temperatures, the Hall conductance as a function of the magnetic field undergoes instead a series of plateaux, where it takes on integer (IQHE) or fractional (FQHE) values, as measured in units of electron charge squared over the Planck constant $\frac{e^2}{h}$. Remarkably, this quantization phenomenon is extremely precise, with accuracy of up to $10^{-8}$, despite of the imprecise characteristics of the material. 

The IQHE by now rather well-understood, see e.g. \cite{Avron,ASZ,Bellissard}, since it is described by non-interacting electron systems. The FQHE, on the other hand, is a strongly-interacting electron system, that makes it extremely challenging problem for physicists and for mathematicians. The most successful approach to FQHE is due to Laughlin \cite{Laughlin} and is based on many-body electronic wave functions defined axiomatically, rather than derived from a particular hamiltonian. Such states are called quantum Hall (QH) states and the best known Laughlin state corresponds to the values of the Hall conductance given by simple fractions $\frac{1}{m}$, where $m\in 2\mathbb{N}_0+1$.

A customary analytic tool central to this paper, is to consider such wave functions on a compact Riemann surface. This idea goes back to Thouless et al., who suggested to use quasi-periodic boundary conditions \cite{ThoulessEtAl} and Haldane--Rezayi \cite{HaldaneRezayi} who constructed Laughlin state on a torus. The key features of the torus models are the following. Firstly, the  QH states are degenerate and finite dimensional. Secondly, they form a vector bundle over the parameter space of Aharonov--Bohm phases, which can be identified with the  Jacobian variety of the torus. This allows to explain a fractional quantization of the Hall conductance:  it turns out to be equal to the slope, i.e. the first Chern class of the bundle divided by its rank \cite{NiuThoulessWu,AvronSeiler,TaoHaldane,KleinSeiler}. We should mention the seminal papers by Avron, Seiler and Zograf on geometric adiabatic transport of the integer QH states on moduli spaces \cite{ASZ,ASZ1}, see also \cite{KMMW} and \cite{SK2015} for a review, and \cite{KlevtsovZvonkine} for a generalization to the Laughlin states. For the physics introduction to QH states on curved backgrounds we refer to \cite{CLW}.

Let us explain our results. A \emph{Wen's matrix} is a symmetric positive definite matrix $K \in \Mat_{g \times g}(\NN_0)$ such that  
\begin{itemize}
\item all entries of the vector $\vec{u} := K^{-1} \vec{e} \in \mathbb{Q}^g$
are positive, where $\vec{e} = \left(\begin{array}{c} 1 \\ \vdots \\ 1\end{array}\right)$; 
\item the diagonal entries of $K$ are either all even or all odd. 
\end{itemize}
In what follows, we put $\delta := \det(K)$. Next, $\rho$ is the sum of the entries of the adjunct matrix $K^\sharp$ (recall that $K^\sharp \in \Mat_{g \times g}(\ZZ)$ and $K^\sharp K = \delta I_g$). We call $K$ \emph{primary} if $\gcd(\delta, \rho) =1$. An important example of a primary Wen's matrix is  $$K_{p, g} = \left(
\begin{array}{cccc}
p+1 & p & \dots & p \\
p & p+1 & \dots & p \\
\vdots & \vdots & \ddots & \vdots \\
p & p & \dots & p+1
\end{array}
\right) \in \mathrm{Mat}_{g \times g}(\NN).$$ In this case we have:    $\delta =  pg +1$ and $\rho = g$.

The multi-layer torus model of  Keski--Vakkuri and Wen \cite{KeskiVakkuriWen} (see also \cite{Wen}) is defined by a datum  $\bigl(E, (K, \vec{\nn})\bigr)$. Here, 
$E = \CC/(\ZZ + \tau \ZZ)$ is a complex torus (where $t = \mathfrak{Im}(\tau) > 0$), $K$ is a Wen's matrix and $\vec\nn \in \NN$ is such that $K \vec\nn = d \vec{e}$ for some $d \in \NN$.

For any $\vec{c} \in \Pi = \Pi_K := K^{-1} \ZZ^g/\ZZ^g$ and 
$\vec{\xi} \in \CC^g$, consider the function $\CC^{n} \stackrel{\Phi_{\vec{c}}}\lar \CC$ depending on $n = n_1 + \dots + n_g$ variables $z^{(k)}_r$ (where $1 \le k \le g$ and  $1 \le r \le n_k$), 
given by the formula
$
\Phi_{\vec{c}} = \Theta\bigl[\vec{c}, \vec{0}\bigr](K\vec{w} + \vec{\xi}|\Omega) \cdot D_{K, \vec{\nn}},
$
where
$$
D_{K, \vec{\nn}} = 
\prod\limits_{k = 1}^g \left(\prod\limits_{1 \le p < q \le n_k} 
\vartheta\left(z_p^{(k)} - z_q^{(k)} \right)\right)^{K_{kk}} \cdot \prod\limits_{ 1 \le k < l \le g} \left(\prod\limits_{p = 1}^{n_k} \prod\limits_{q = 1}^{n_l}
\vartheta\left(z_p^{(k)} - z_q^{(l)} \right)\right)^{K_{kl}}.
$$
Here, 
$\Omega = \tau K$, $\vec{w} = \left(\begin{array}{c} z_1^{(1)} + \dots + z_{n_1}^{(1)} \\ \vdots \\  z_1^{(g)} + \dots + z_{n_g}^{(g)}\end{array} \right)$, $\vartheta(z)$ is the first theta-function of Jacobi  and $\Theta[\vec{a}, \vec{b}](\vec{w} \big| \Omega)$ is the multi-variate theta function of Riemann \cite{Mumford} 
(with $\vec{a}, \vec{b} \in \RR^g$).
For $K = K_{p,g}$ we have:  
$$
D_{K, \vec\nn} = 
\left[\prod\limits_{k = 1}^g \prod\limits_{1 \le s < t \le n} \vartheta(z^{(k)}_s - z^{(k)}_t)\right]^{p+1}
\left[\prod\limits_{1 \le k < l \le g}^g \prod\limits_{s, t = 1}^n \vartheta(z^{(k)}_s - z^{(l)}_t)\right]^{p}.
$$
For the fermionic statistic, the parameter $p$ is required to be even and for the bosonic statistic $p$  is odd.

As it will be explained in the paper, the  functions $\bigl(\Phi_{\vec{c}}\bigr)_{\vec{c} \in \Pi}$, introduced  by Keski-Vakkuri and Wen in \cite{KeskiVakkuriWen}, belong to the ground state $\mathsf{W}_{K, \vec{\nn}, \vec{\xi}}$ \,of an appropriate many-body magnetic Schr\"odinger operator without interactions and their linear span $\mathsf{V}_{K, \vec{\nn}, \vec{\xi}}$ forms a proper linear subspace that can be axiomatically characterized; see Lemma \ref{L:HRwavefunction}. In the singe-layer case $g = 1$, 
these functions coincide with the wave functions of Haldane and Rezayi \cite{HaldaneRezayi}.

\smallskip
\noindent
The  main results of this work are the following.
\begin{itemize}
\item We construct a holomorphic  vector bundle $\kB$ on  $A = \CC^g/(\ZZ^g + \tau \ZZ^g)$ (called \emph{magnetic bundle} of the multi-layer torus model) for which there are ``natural'' isomorphisms  $\kB\Big|_{\bigl[\vec\xi\bigr]} \cong \mathsf{V}_{K, \vec{\nn}, \vec\xi}$ for any  $\vec\xi \in \CC^g$; see
Theorem \ref{T:MagneticBundle}.
\item We show that $\End(\kB) = \CC$ and the  first Chern class of $\kB$ is given by the formula
\begin{equation*}
c_1(\kB) = -\frac{i}{2t} \sum\limits_{p, q = 1}^g K^\sharp_{pq} d\xi_{p} \wedge d\bar{\xi}_q,
\end{equation*}
where $K^\sharp$ is the adjunct matrix of $K$. 
Moreover, the vector bundle $\kB$ is  semi-homogeneous; see 
Theorem \ref{T:MagneticBundle}. This, in particular,  implies that the total Chern class of $\kB$ can be expressed via its first Chern class: 
$c(\kB) = \left(1 + \dfrac{c_1(\kB)}{\delta}\right)^\delta.$
\item The analytic type of $\kB$ is determined by the matrix $K$ and does not depend on the choice  of $\vec{\nn}$ (i.e.~it is independent on the number of particles of our system). However, $\kB$ can be equipped with  a natural hermitian metric, arising from the embedding $\kB\Big|_{\bigl[\vec\xi\bigr]} \cong \mathsf{V}_{K, \vec{\nn}, \vec\xi} \subset \mathsf{W}_{K, \vec{\nn}, \vec\xi}$ (in other words, this metric is defined by the scalar product of  the Hilbert space of the underlying many-body problem), which depends on $\vec{\nn}$.
\item If $K$ is primary  (e.g.~$K = K_{p, g}$) then we prove that 
$\bigl(\Phi_{\vec{c}}\bigr)_{\vec{c} \in \Pi}$ is an orthogonal basis of the vector space $\mathsf{V}_{K, \vec{\nn}, \vec\xi}$ and the norms of all its basis vectors are equal; see Theorem \ref{T:GramMatrixKVWSpace}.
This implies that  the canonical Bott--Chern connection of the hermitian holomorphic vector bundle $\kB$ is projectively flat; see Proposition \ref{P:ProjFlat}. There exists another natural choice of a hermitian metric on $\kB$ (coming from the dynamics of the center of mass), with respect to which the corresponding  Bott--Chern connection is projectively flat for an \emph{arbitrary} Wen's matrix  $K$; see Theorem \ref{T:HeisenbergGroupRepr} 
and Remark \ref{R:ProjFlat}. 
\item The \emph{restricted magnetic bundle} $\kU$ is the restriction of $\kB$ on $E$ with respect to the diagonal embedding $E = \CC/(\ZZ + \tau \ZZ) \lar \CC^g/(\ZZ^g + \tau \ZZ^g) = A$. We show that $\kU$ has rank $\delta$ and degree $-\rho$. In particular, its slope is 
\begin{equation*}
\frac{\mathsf{deg}(\kU)}{\mathsf{rk}(\kU)}  = - \dfrac{\rho}{\delta}.
\end{equation*}
Moreover, $\kU$ is stable if and only if $K$ is primary; see Theorem \ref{T:MagneticBundleCurve}. 
In particular, for $K = K_{p, g}$  the absolute value of the slope of $\kU$  is a Jain fraction  \cite{Jain}: 
\begin{equation*}
\left|\frac{\mathsf{deg}(\kU)}{\mathsf{rk}(\kU)} \right| = \dfrac{g}{gp +1},
\end{equation*}
see Example \ref{E:JainWen} and Remark \ref{R:Jain}.
\end{itemize}

In the case of the IQHE, Avron, Seiler and Simon \cite{AvronSeilerSimon} interpreted the Hall conductance as the  (properly rescaled) first Chern class of a magnetic line bundle (the physical derivation for the underlying expression of the Hall conductance was obtained by Thouless et al.~in \cite{ThoulessEtAl}); see also \cite{Kohmoto} for further elaborations. We refer to \cite{KleinSeiler} for the interpretation of the fractional quantization of the Hall conductance in the FQHE via the slope of the restricted magnetic vector bundle;  see also  \cite[formula (3.1)]{NiuThoulessWu}. Thus, the multi-layer torus model based on $K$-matrices of type  $K_{p, g}$ corresponds to the main series of Hall fractions which were experimentally observed. 

The first attempt to include  the works \cite{HaldaneRezayi, KeskiVakkuriWen} into an algebro-geometric framework was made by Varnhagen \cite{Varnhagen}. This approach was continued by  Guilarte, Mu\~noz Porras and de la Torre Mayado \cite{Salamanca}. In that important work,  the usage of Fourier--Mukai transforms in condensed matter physics was in particular pioneered, years before this technique became to be  established tool in the string theory! First steps towards the algebro-geometric study of multi-layer models were made in a subsequent work \cite{Salamanca2}. Unfortunately, both papers \cite{Salamanca, Salamanca2} did not get appropriate attention neither in the mathematics   nor in the physics   literature (partially due to the fact that mathematical tools of that works appeared  at a time to be too uncommon for the physics community, whereas  the mathematical exposition was lacking clarity and precision in details).

On the physics side, the original interest in the model of Keski-Vakkuri and Wen stemmed from its relevance to the experiments in a GaAs/AlGaAs heterostructure forming a double quantum well, where novel correlated states can emerge when the inter-layer and intra-layer Coulomb interactions are comparable to each other, see \cite{KeskiVakkuriWen} and references therein. 
More recently, with the discovery that in bilayer graphene the electronic flat bands can occur, it was suggested that the electron spin and graphene-valley isospin degrees of freedom may play the role of the layers of the Keski-Vakkuri-Wen model. This has recently led to the revival of interest in the multi-layer model, which we shall not attempt to review here and instead refer the reader to \cite{MagicAngles}.

The plan of the paper is the following. In Section \ref{S:ThetaFunctions} we review (mainly following \cite{Mumford}) some standard results on theta-functions of one complex variable, regarding them as holomorphic sections of appropriate line bundles on elliptic curves. We also discuss the Landau problem on a complex torus and explain that the ground state of a one-particle magnetic Schr\"odinger operator coincides with the space of holomorphic sections of a holomorphic line bundle. Here the finite Heisenberg groups acting on the ground state plays an important role. 

For the reader's convenience, we first discuss in Section \ref{S:HaldaneRezayi} the single-layer torus model of Haldane and Rezayi \cite{HaldaneRezayi}, before going to multi-layer case. The main results of this Section are Lemma \ref{L:HRSpace} (giving an ``axiomatic'' characterization of the vector space of these functions) and Theorem \ref{T:GramMatrixHRSpace} about the Gram matrix of the distinguished basis of this space. 

In Section \ref{S:CenterMass} we study the space of global holomorphic sections of a certain line bundle on a product of elliptic curves. This space admits an action of a generalized Heisenberg group and we construct in Theorem \ref{T:HeisenbergGroupRepr} its distinguished basis. 

An algebro-geometric characterization of the space of wave functions of Keski-Vakkuri and Wen \cite{KeskiVakkuriWen} is given in Section \ref{S:WaveFunctions};   see in particular Remark \ref{R:CenterMassAllWave}. An important result is that  two apriori different actions by magnetic translations on the space $\mathsf{V}_{K, \vec\nn, \vec\xi}$ of wave functions of Keski-Vakkuri and Wen coincide; see Lemma \ref{L:Magnetic}.

In Section 
\ref{S:FMT}, we recall the theory of vector bundles on abelian varieties in the context of Fourier--Mukai transforms. The main results here are Theorem \ref{T:FMTCenterMass} about an explicit description of the dual isogeny induced by the matrix $K$ and properties of the Fourier--Mukai transforms of certain distinguished  line bundles. 

In Section \ref{S:Multilayer}, we define the magnetic vector bundle $\kB$ on $\CC^g/(\ZZ^g + \tau \ZZ^g)$ attached to a datum $\bigl(E, (K, \vec{\nn})\bigr)$ and explore its basic properties (see Theorem \ref{T:MagneticBundle}). The holomorphic vector bundle $\kB$ carries two natural hermitian metrics: the one determined  by the many-body dynamics and the one given by the center-of-mass dynamics. The corresponding  Bott--Chern connection is always projectively flat with respect to the second metric as well as with respect to the first metric under the assumption that Wen's matrix $K$ is primary. 

Finally, in Section \ref{S:RMultilayer},  we study properties of the restricted magnetic bundle $\kU$; the main results are  given by Theorem  \ref{T:MagneticBundleCurve}. 

\smallskip
\noindent
\emph{List of notation}. For convenience of the reader we introduce now the most important
notation used in this paper.
\begin{itemize}
\item We fix $\tau \in \CC$ with $t = \mathfrak{Im}(\tau) > 0$. Next, $E = \CC/\langle 1, \tau\rangle$ denotes the corresponding complex torus. For $a, b \in \RR$ we set $\xi = a\tau + b \in \CC$. Abusing the notation, we frequently identify $\xi \in \CC$ with the corresponding image $[\xi] \in E$. For $z \in \CC$ we use its presentations $z = u + i v = x + \tau y$ with $u, v, x, y \in \RR$. 
\item For any $k \in \ZZ$ and $\xi \in E$, we denote by $\kL_{k, \xi}$ the holomorphic line bundle on $E$ of degree $k$ and continuous parameter $\xi$. It carries a natural hermitian metric denoted by $h$. For $k > 0$, we denote by $\mathsf{W}_{k, \xi}$ the space of holomorphic sections of 
$\kL_{k, \xi}$. It  has a distinguished basis $(h_1, \dots, h_k)$ given by  theta-functions $h_j(z) = \vartheta\Bigl[\frac{j-1}{k}, 0\Bigr](kz + \xi | k \tau)$. We also denote by $\vartheta(z) :=  \vartheta\left[\frac{1}{2}, \frac{1}{2}\right](z | \tau)$ the first theta-function of Jacobi. 
Next, $\mathsf{B}_{k, \xi}$ (respectively $\mathsf{H}_{k, \xi}$) denotes the space of smooth (respectively square integrable) global sections of $\kL_{k, \xi}$. We also put  $\kL_{k, \xi}^\sharp = \kL_{k, \xi} \otimes \kS$, where $\kS$ is a distinguished spin bundle on $E$ and denote by $\mathsf{W}^\sharp_{k, \xi}$, $\mathsf{B}^\sharp_{k, \xi}$ and $\mathsf{H}^\sharp_{k, \xi}$ the corresponding spaces of holomorphic, smooth and square integrable global sections. 
\item $T_1$ and $T_2$ are magnetic translations, acting on the spaces $\mathsf{W}_{k, \xi}$, $\mathsf{H}_{k, \xi}$ etc. 
\item In this work, $g \in \NN$ always denotes the number of layers of the multi-layer torus model of the FQHE. The notation is chosen in such a way that the single- and multi-layer cases match. Elements of $\CC^g$ are denoted by $\vec{z}$, $\vec{\xi}$ etc. We denote by  $\Theta[\vec{a}, \vec{b}](\vec{z}\,|\, \Omega)$ the multi-variate theta function of Riemann with characteristics $\vec{a}, \vec{b} \in \RR^g$, defined by a symmetric matrix $\Omega$ with positive definite imaginary part. 
\item 
A Wen datum is a pair $(K, \vec\nn)$, where  $K \in \mathsf{Mat}_{g \times g}(\NN_0)$ is a positive definite matrix and $\vec{\nn} = (n_1, \dots, n_g)$ is a tuple, whose entries are the numbers of particles in each layer (with further properties to be satisfied). We put:  $\delta = \det(K)$ and
$\Pi_K = K^{-1} \ZZ^g/\ZZ^g$. Next,  $G_K$ is the finite Heisenberg group determined by $K$,  $n = n_1 + \dots + n_g$ is the entire number of particles of our system and $\Omega = \tau K$. 
\item We set $A = \CC^g/(\ZZ^g + \tau \ZZ^g)$, $B = \CC^g/(\ZZ^g + \tau K \ZZ^g)$ and
$X := 
\underbrace{E \times \dots \times E}_{n  \; \mbox{\scriptsize{\sl times}}}$. Next, we consider a distinguished isogeny 
$A \stackrel{\kappa}\lar B, [\vec{z}] \mapsto [K \vec{z}]$. 
\item $\mathsf{V}_{K, \vec{\nn}, \vec{\xi}}$ is the space of wave functions of Keski-Vekkuri and Wen, which is a subspace of the ground state $\mathsf{W}_{K, \vec{\nn}, \vec{\xi}}$ of the underlying many-body problem. The space $\mathsf{V}_{K, \vec{\nn}, \vec{\xi}}$ has a distinguished basis $\bigl(\Phi_{\vec{c}}\bigr)_{\vec{c} \in \Pi}$ given in terms of theta-functions. 
\item For a complex algebraic variety $Y$ we denote by $\Coh(Y)$ its category of coherent sheaves and by $D^b\bigl(\Coh(Y)\bigr)$ the corresponding derived category. For a morphism $Y \stackrel{f}\lar Y'$ of complex algebraic varieties, we denote by  
$\Coh(Y) \stackrel{f_\ast}\lar \Coh(Y')$ and $\Coh(Y') \stackrel{f^\ast}\lar \Coh(Y)$ 
the  classical direct and inverse image functors, respectively, whereas 
$D^b\bigl(\Coh(Y)\bigr) \stackrel{Rf_\ast}\lar D^b\bigl(\Coh(Y')\bigr)$ denotes the derived direct image functor. In the case   $f^\ast$ is exact,  we use the same notation $f^\ast$ for the corresponding derived functor $D^b\bigl(\Coh(Y')\bigr) \stackrel{f^\ast}\lar D^b\bigl(\Coh(Y)\bigr)$.
 
\item For a complex torus  $T$ we denote by $\widehat{T}$ the corresponding dual torus. 
Calligraphic letters like $\kL, \kR, \kF$  etc.~are always used for  vector bundles/coherent sheaves. In particular, $\kB$ denotes the ``magnetic bundle'', $\kQ$  the Poincar\'e line bundle for a general abelian variety $T$ and $\kP$ the Poincar\'e line bundle for $E$. For an abelian variety $T$, $D^b\bigl(\Coh(T)\bigr) \stackrel{\FF_T}\lar D^b\bigl(\Coh(\widehat{T})\bigr)$ is the corresponding Fourier--Mukai transform. 

\item If $Y_1$ and $Y_2$  are complex varieties and   $\kF_i$ is a coherent sheaf on $Y_i$ for $i = 1, 2$ then 
$\kF_1 \boxtimes \kF_2 = \pi_1^\ast(\kF_1) \otimes \pi_2^\ast(\kF_2)$, where $Y_1 \times Y_2 \stackrel{\pi_i}\lar Y_i$ is the canonical projection. 
\end{itemize}

\section{Holomorphic line bundles on complex tori 
and  theta-functions}\label{S:ThetaFunctions}

\noindent
 Let $\tau \in \CC$ be such that $t = \mathfrak{Im}(\tau) > 0$, $\Lambda = \langle 1, \tau\rangle_{\mathbbm{Z}} \subset \CC$ be the corresponding lattice and  $E := \CC/\Lambda$ be the corresponding complex torus. Recall that the theta-function with characteristics $a, b \in \RR$ is defined via the  series:
\begin{equation}\label{E:ThetaWithCharacteristics}
\vartheta[a, b](z | \tau):= \sum\limits_{k \in \ZZ} 
\exp\bigl(\pi i \tau (k+a)^2+2\pi i (k+a)(z+b)\bigr).
\end{equation}
We have the following classical results, see e.g.~\cite[Chapter I]{Mumford}.
\begin{proposition}\label{P:Basics} As a function of $z$, the theta-series (\ref{E:ThetaWithCharacteristics}) converges absolutely and uniformly {on compact subsets of} $\CC$ and satisfies the following quasi-periodic conditions:
\begin{equation}\label{E:ThetaFunctionsTransfRules}
\left\{
\begin{array}{lcl}
\vartheta[a, b](z + 1 | \tau ) & = & \exp(2\pi i a) \vartheta[a, b](z | \tau) \\
\vartheta[a, b](z + \tau | \tau) & = & \exp(- 2\pi i (z +b) - \pi i \tau) \vartheta[a, b](z | \tau).
\end{array}
\right.
\end{equation}
Modulo the lattice $\Lambda$, the function $\vartheta[a, b](z | \tau)$ has a unique  zero at the point 
\begin{equation*}
p_\xi:= \sigma - \xi,\; \mbox{\rm where} \; \sigma = \dfrac{1 + \tau}{2} \, \mbox{\rm and} \; \xi = a \tau + b.
\end{equation*} Moreover, this zero is simple. 
In particular, the function
\begin{equation}\label{E:ThetaOdd}
\vartheta(z) :=  \vartheta\left[\frac{1}{2}, \frac{1}{2}\right](z | \tau)
\end{equation}
has a unique simple zero at $z = 0$ modulo $\Lambda$. Moreover, it is odd: $\vartheta(-z) = - \vartheta(z)$. 
\end{proposition}

\smallskip
\noindent
Recall (see e.g.~\cite{BirkenhakeLange}) that an 
\emph{automorphy factor} is a pair
$(\Upsilon, V)$, where $V$ is a finite dimensional vector space over $\CC$ and
$
\Upsilon: \Lambda \times \CC \lar \GL(V)
$
is a holomorphic function such that $\Upsilon(\lambda + \mu, z) = \Upsilon(\lambda, z + \mu) \Upsilon(\mu, z)$ for all
$\lambda, \mu \in \Lambda$ and $z \in \CC$. Such a pair defines the following  holomorphic vector bundle on
the torus $E$:
\begin{equation*}
\kE(\Upsilon, V) := \CC \times V/\sim, \;
\mbox{where} \; \;  (z, v) \sim \bigl(z + \lambda, \Upsilon(\lambda, z) v\bigr) \;\; 
\mbox{for all} \;\;  (\lambda, z, v) \in \Lambda \times \CC \times V.
\end{equation*}
Two such vector bundles $\kE(\Upsilon', V)$ and $\kE(\Upsilon'', V)$ are isomorphic if and only if
there exists a holomorphic function $\Xi: \CC \rightarrow \GL(V)$ such that
$$
\Upsilon''(\lambda, z) = \Xi(z + \lambda) \Upsilon'(\lambda, z) \Xi(z)^{-1} \quad
\mbox{for all} \quad (\lambda, z) \in \Lambda \times \CC.
$$
Now we consider the special case of line bundles, here we have $V = \CC$. For an automorphy factor $\Upsilon$ we put $\upsilon_1(z) = \Upsilon(1, z)$ and $\upsilon_\tau(z) = \Upsilon(\tau, z)$. Obviously, $\Upsilon$ is specified by the pair $(\upsilon_1, \upsilon_\tau)$ and we shall use the notation
$
\kL(\upsilon_1, \upsilon_\tau) := \kE(\Upsilon, \CC).
$

\smallskip
\noindent
Recall the following well-known results; see e.g.~\cite{Mumford, BK4}.
\begin{proposition}\label{P:LibeBundlesTorus}  In the notation as above, the  following statements are true. 
\begin{itemize}
\item For any $a, b \in \RR$ we have:
\begin{equation}\label{E:LBTorus}
\kL\bigl(\exp(2\pi i a), \exp(-2\pi i b)\bigr) \cong 
\kL(1, \exp(-2\pi i \xi)\bigr) \cong \kO\bigl([0]-[\xi]\bigr),
\end{equation}
where $\xi = a \tau + b$.
\item Let $\varphi (z) := \exp(-\pi i \tau - 2 \pi i z)$. Then we have: 
\begin{equation}
\kL\bigl(1, \exp(-2\pi i \xi)\varphi(z)\bigr) \cong \kO\bigl([p_\xi]\bigr).
\end{equation}
In particular, the  space of holomorphic sections of   $\kL\bigl(1, \exp(-2\pi i \xi)\varphi(z)\bigr)$ is one-dimensional and generated by $\vartheta[0, 0](z + \xi | \tau)$. 
\item The set of the isomorphism classes of holomorphic line bundles on $E$ of degree $k \in \ZZ$ can be described as $\kL_{k, \xi} := 
\kL\bigl(1, \exp(-2\pi i \xi)\varphi(z)^k\bigr)$, with $\xi = a \tau + b$  and $a, b \in \RR/\ZZ$. For $k > 0$, the space $\mathsf{W}_{k, \xi}:= \Gamma\bigl(E, \kL_{k, \xi})$ of holomorphic sections of the line bundle $\kL_{k, \xi}$ has dimension $k$ and has a distinguished basis $(h_1, \dots, h_k)$, where
\begin{equation}\label{E:BasisThetaFunctions}
h_j(z) = \vartheta\Bigl[\frac{j-1}{k}, 0\Bigr](kz + \xi | k \tau) \; \mbox{\rm for} \; 1 \le j \le k.
\end{equation}
\end{itemize}
\end{proposition}

\begin{remark} In what follows, we shall denote 
$
\kS:= \kL(-1, -1) \cong \kO\bigl([0] -[\sigma]\bigr) 
$
and $$\kL_{k, \xi}^\sharp := \kL_{k, \xi} \otimes \kS = \kL(-1, - \exp(-2\pi i \xi)\varphi(z)^k\bigr) \cong \kL_{k, \xi^\sharp}$$ with $\xi^\sharp = \xi + \sigma$.
\end{remark}

\begin{remark} The  space of \emph{smooth} global sections of  $\kL_{k, \xi}$ has the following description: 
\begin{equation}\label{E:SmoothSections}
\mathsf{B}_{k, \xi} := \left\{
\CC \stackrel{f}\lar \CC \left| \, \begin{array}{l}
f \; \mbox{\rm is smooth}, f(z+1) = f(z) \\
f(z + \tau) = \exp(-2\pi i \xi) \varphi(z)^k f(z)
\end{array}
\right.
\right\}.
\end{equation}
In a similar vain, the space of smooth global sections of the line bundle $\kL_{k, \xi}^\sharp$ is
\begin{equation*}
\mathsf{B}_{k, \xi}^\sharp := \left\{
\CC \stackrel{f}\lar \CC \left| \, \begin{array}{l}
f \; \mbox{\rm is smooth}, f(z+1) = -f(z) \\
f(z + \tau) = -\exp(-2\pi i \xi) \varphi(z)^k f(z)
\end{array}
\right.
\right\}.
\end{equation*}
\end{remark}

\smallskip
\noindent
Let $z = x + \tau y$, where $x, y\in \RR$. For any $k \in \ZZ$
and $\xi = a \tau + b$, we consider the function
\begin{equation}\label{E:MetricLB1}
h(z) = h_{k, \xi}(x, y) := \exp(- 2 \pi kty^2  - 4 \pi a ty).
\end{equation}

\begin{lemma}\label{L:MetricOnLB}
For any $f \in \mathsf{B}_{k, \xi}$ or $f \in \mathsf{B}_{k, \xi}^\sharp$, we have a smooth function
\begin{equation}\label{E:NormSection}
E \stackrel{\eta_f}\lar {\RR_{\geqslant0}}, \; z  \mapsto \eta_f (z) = \lVert f(z) \rVert^2 := |f(z)|^2 \,h(z).
\end{equation}
In this way, the complex line bundles $\kL_{k, \xi}$ and $\kL_{k, \xi}^\sharp$ on $E$ get equipped with the corresponding  hermitian metrics. In particular, we obtain hermitian scalar products on the spaces of smooth global sections  $\mathsf{B}_{k, \xi}$ and $\mathsf{B}_{k, \xi}^\sharp$ given by the formula
\begin{equation}\label{E:scalarproduct}
\langle f, g \rangle := \iint\limits_{[0, 1]^2} \exp(-2\pi k ty^2 - 4 \pi aty)
f(x, y) \overline{g(x,y)}   dx \wedge dy.
\end{equation}
\end{lemma}

\begin{proof} The result follows from the facts that 
$$
h(x+1, y) = h(x,y) \; \mbox{and} \; h(x, y+1) = \exp(- 2 \pi t(2 ky + 2 a + k)) h(x, y)
$$ 
combined with the quasi-periodicity properties of $f$. \end{proof}

\begin{remark} In a standard manner, we can also introduce the Hilbert space of square integrable global sections of the Hermitian line bundle $\kL_{k, \xi}$: 
\begin{equation}\label{E:Hilbert space}
\mathsf{H}_{k, \xi} := \left\{
\CC \stackrel{f}\lar \CC \left| \, \begin{array}{l}
 f(z+1) = f(z), f(z + \tau) = \exp(-2\pi i \xi) \varphi(z)^k f(z) \\
\iint\limits_{[0, 1]^2} \exp(-2\pi k ty^2 - 4 \pi aty)
\big|f(x, y)\big|^2 dx \wedge dy < \infty
\end{array}
\right.
\right\}.
\end{equation}
\end{remark}

\begin{lemma} We write $z = u + i v \in \CC$ with $u, v \in \RR$ and $\partial = \frac{1}{2}(\partial_u - i \partial_v)$, $\bar\partial = \frac{1}{2}(\partial_u + i \partial_v)$. Then we have a $\CC$-linear map $\bar\partial: \mathsf{B}_{k, \xi} \lar 
\mathsf{B}_{k, \xi}$, which admits an adjoint operator given by the formula
$\bar\partial^\ast = - \left(\partial + \dfrac{\partial(h)}{h}\right)$. The same result is true for the space $\mathsf{B}_{k, \xi}^\sharp$.
\end{lemma}

\begin{proof} Since $\bar\partial(\varphi) = 0$, the differential operator $\bar\partial$ is indeed an endomorphism of $\mathsf{B}_{k, \xi}$. It is a straightforward computation that $\partial + \dfrac{\partial(h)}{h}$ also maps $\mathsf{B}_{k, \xi}$ to itself. Finally, it remains to check that 
\begin{equation}\label{E:equationadjoint}
\big\langle \bar\partial f, g \big\rangle + \Big\langle f, \left(\partial + \dfrac{\partial(h)}{h}\right) g \Big\rangle = 0
\end{equation}
for all $f, g \in \mathsf{B}_{k, \xi}$. Since $\overline{\partial g} = \bar\partial \bar{g}$ and $h$ is real-valued, the left hand side of (\ref{E:equationadjoint}) takes the form
$$
\iint\limits_{[0, 1]^2} h  \bigl(\bar{\partial} f) \bar{g} \, dx \wedge dy + 
\iint\limits_{[0, 1]^2} h f \overline{\left(\partial(g) + g \dfrac{\partial(h)}{h} \right)} dx \wedge dy  = 
\iint\limits_{[0, 1]^2} \bar{\partial}\bigl(h f \bar{g}\bigr) dx \wedge dy.
$$
The latter integral can easily be shown to be zero. 
\end{proof}

\smallskip
\noindent
\textbf{Postulate}. The quantum mechanical motion of a charged particle on the torus
$E$ in the presence of a  constant magnetic field of strength $k \in \ZZ$  is described (after an appropriate rescaling) by the following  self-adjoint  operator 
\begin{equation}\label{E:MagneticSchroedingerTorus}
\Delta^{(k)} := - \left(\partial + \dfrac{\partial(h_{k, \xi})}{h_{k, \xi}}\right) \bar{\partial}: \; \mathsf{H}_{k, \xi} \lar \mathsf{H}_{k, \xi};
\end{equation}
see for instance \cite[Proposition 2]{Prieto} and the references therein.

\begin{remark} 
The  magnetic Schr\"odinger operator $\Delta^{(k)}$ is explicitly given by the formula  $$\Delta^{(k)}  = -\dfrac{1}{4}\bigl(\partial_u^2 + \partial_v^2 \bigr) - 2\pi i \dfrac{kv}{t}\bar\partial.$$ 
Moreover, it  essentially coincides with the Bochner Laplacian of the hermitian holomorphic line bundle $\bigl(\kL_{k,\xi}, h_{k, \xi}\bigr)$, see e.g.~\cite{Prieto, MaMarinescu}. For $k > 0$, the ground state of $\Delta^{(k)}$ is the space of holomorphic sections of the line bundle $\kL_{k, \xi}$. By Proposition \ref{P:LibeBundlesTorus}, the ground state of $\Delta^{(k)}$  is $k$-dimensional with a distinguished basis given by (\ref{E:BasisThetaFunctions}). Since 
$\bigl[\bar\partial, \bar\partial^\ast\bigr] = \dfrac{\pi k}{t}$, it follows that the spectrum of $\Delta^{(k)}$ is the set $\left\{ \left. \dfrac{2\pi k}{t} n \, \right| \,  n \in \NN_0 \right\}$.
\end{remark}

\begin{lemma}\label{L:MagneticTranslations}
We have a pair of \emph{unitary} operators $T_1, T_2: \mathsf{H}_{k, \xi} \lar \mathsf{H}_{k, \xi}$ (called \emph{magnetic translations}) given for any $f \in \mathsf{H}_{k, \xi}$ by the formulae:
\begin{equation}\label{E:MagneticTranslationsFormulae}
\left\{
\begin{array}{l}
\bigl(T_1(f)\bigr)(z) = f\Bigl(z + \dfrac{1}{k}\Bigr) \\
\bigl(T_2(f)\bigr)(z) = \exp\Bigl(\dfrac{2\pi i \xi + \pi i \tau}{k}\Bigr)\exp(2\pi iz )f\Bigl(z + \dfrac{\tau}{k}\Bigr).
\end{array}
\right.
\end{equation}
These operators satisfy the following relations:
\begin{equation}\label{E:RelationsTranslations}
\left\{
\begin{array}{l}
T_1^k = \mathsf{Id} = T_2^k \\
T_1 T_2 = q T_2 T_1,
\end{array}
\right.
\end{equation}
where $q = \exp\bigl(\frac{2\pi i}{k}\bigr)$. Moreover, for $k > 0$ the linear operators $T_1, T_2: \mathsf{W}_{k, \xi} \lar \mathsf{W}_{k, \xi}$ are given in  by the  matrices:
\begin{equation}\label{E:HeisenbergRepresentation}
\bigl[T_1\bigr] = \left(
\begin{array}{cccc}
1 & 0 & \dots & 0 \\
0 & q & \dots & 0 \\
\vdots & \vdots& \ddots & \vdots \\
0 & 0 & \dots & q^{k-1}
\end{array}
\right) \quad \mbox{\rm and} \quad
\bigl[T_2\bigr] = \left(
\begin{array}{cccc}
0 & \dots & 0 & 1 \\
1 & \dots & 0 & 0 \\
\vdots & \ddots& \vdots & \vdots \\
0 & \dots & 1 & 0
\end{array}
\right)
\end{equation}
with respect to the distinguished basis $(h_1, \dots, h_k)$ of $\mathsf{W}_{k, \xi}$ given by
(\ref{E:BasisThetaFunctions}).
\end{lemma}

\begin{proof} It is a straightforward check that $T_1$ and $T_2$ are endomorphisms of  $\mathsf{H}_{k, \xi}$ and satisfy the relations (\ref{E:RelationsTranslations}). To prove their unitarity, observe that for any $f \in \mathsf{H}_{k, \xi}$ we have:
$$
\left\{\begin{array}{ccc}
\eta_{T_1(f)}(x, y) = \eta_f\left(x + \dfrac{1}{k}, y\right) \\
\eta_{T_2(f)}(x, y) = \eta_f\left(x, y + \dfrac{1}{k}\right),
\end{array}
\right.
$$
where the function $\eta_f: E \lar \RR$ is given by (\ref{E:NormSection}). It follows that
$$
\big\lVert T_l(f) \big\rVert = \iint\limits_{\RR^2/\ZZ^2} \eta_{T_l(f)} dx \wedge dy = \iint\limits_{\RR^2/\ZZ^2} \eta_{f} dx \wedge dy  = \lVert f \rVert
$$
for $l = 1, 2$, implying  that the operator $T_l$ is indeed unitary. The formulae for the action of $T_1$ and $T_2$ on the basis $(h_1, \dots, h_k)$ of $\mathsf{W}_{k, \xi}$
follow from  (\ref{E:ThetaFunctionsTransfRules}). 
\end{proof}

\begin{remark}
It is easy to see that $\bigl[\bar\partial, T_l\bigr] = 0$ for $l = 1, 2$. Since the operator $T_l$ is unitary, it also commutes with $\bar\partial^\ast$ and, as a consequence, with the magnetic Schr\"odinger operator $\Delta^{(k)}$ given by (\ref{E:MagneticSchroedingerTorus}).
\end{remark}

\begin{definition}\label{D:HeisenbergGroup} For any $k \in \NN$, let  $\Gamma_k = \bigl\langle \exp(\frac{2\pi i}{k}) \bigr\rangle \subset \CC^\ast$.
The Heisenberg group $G_k$ is the set $(\ZZ_k \times \ZZ_k) \times \Gamma_k$ equipped with the group law given by the rule
\begin{equation}
\bigl(([a_1], [b_1]), \gamma_1\bigr) \cdot \bigl(([a_2], [b_2]), \gamma_2\bigr) = 
\left(\bigl([a_1 + a_2], [b_1 + b_2]\bigr), \exp\left(\frac{2\pi i}{k} a_1 b_2\right) \gamma_1 \gamma_2\right).
\end{equation}
Note that $G_k$ is a central extension of $\ZZ_k \times \ZZ_k$ by $\Gamma_k$.
Moreover, for any $n \in \NN$ we have a surjective group homomorphism
\begin{equation}
G_{kn} \rightarrowdbl G_k, \quad 
\bigl((a \; \mod \; kn, b \; \mod \; kn), \gamma\bigr) \mapsto \bigl((a \; \mod \; k, b\;  \mod \; k), \gamma^n\bigr).
\end{equation}
\end{definition}

\begin{lemma} For any $k \in \NN$, the assignment  
\begin{equation}\label{E:ActionHeisenberg}
\bigl((\bar{1}, \bar{0}), 1 \bigr) \mapsto T_1,
\quad 
\bigl((\bar{0}, \bar{1}), 1 \bigr)
\mapsto T_2 \quad \mbox{\rm and} \quad 
\bigl((\bar{0}, \bar{0}), \gamma \bigr)
 \mapsto \mathsf{mult}_\gamma
\end{equation}
defines  a unitary representation of the Heisenberg group $G_k$ on  the Hilbert space $\mathsf{H}_{k, \xi}$ and its finite dimensional  subspace $\mathsf{W}_{k, \xi}$.  Moreover, the  representation of $G_k$ on 
$\mathsf{W}_{k, \xi}$ is irreducible.
\end{lemma}
\begin{proof} The fact that (\ref{E:ActionHeisenberg}) defines  a unitary representation of the group $G_k$ follows from 
Lemma \ref{L:MagneticTranslations}. A more general statement (including the irreducibility of $\mathsf{W}_{k, \xi}$) will be proven in Theorem \ref{T:HeisenbergGroupRepr}.
\end{proof}

\begin{remark}\label{R:MagneticTranslations}
In a similar fashion, we have  unitary operators $T_1, T_2: \mathsf{H}_{k, \xi}^\sharp \lar \mathsf{H}_{k, \xi}^\sharp$, given by the \emph{same formulae} (\ref{E:MagneticTranslationsFormulae}). They satisfy the relations 
\begin{equation}
\left\{
\begin{array}{l}
T_1^k = - \mathsf{Id} = T_2^k \\
T_1 T_2 = q T_2 T_1.
\end{array}
\right.
\end{equation}
 It follows that  $\mathsf{H}_{k, \xi}^\sharp$ (as well as  $\mathsf{W}_{k, \xi}^\sharp$) is the  representation space  of an appropriate   central extension $\widetilde{G}_k$ of the Heisenberg group $G_k$:
\begin{equation}
0 \lar Z \lar \widetilde{G}_k \lar G_k \lar 1,
\end{equation} 
where $Z = \left\langle c \, \big| \, c^2 = e\right\rangle \cong \ZZ_2$ and $c$ acts on $\mathsf{H}_{k, \xi}^\sharp$ as $-\mathsf{Id}$. 
\end{remark}

\section{Haldane--Rezayi wave functions on a torus}\label{S:HaldaneRezayi}
To define the model, we  fix two natural numbers: $n$ (the number of particles) and $m$ (this number reflects the statistic of the system and the strength of the magnetic field on $E$) and put
\begin{equation}\label{E:HRCentral1}
D_m(z_1, \dots, z_n) := \left(\prod\limits_{p < q} \vartheta(z_p - z_q)\right)^m,
\end{equation}
where $\vartheta(z)$ is the odd theta-function of Jacobi (\ref{E:ThetaOdd}). Recall that $\vartheta(z)$ has a unique simple zero at $z = 0 \; \mathsf{mod}\;  \Lambda$ (see Proposition \ref{P:Basics}). 
For any $1 \le j \le m$ and $\xi = a \tau + b$ we denote
\begin{equation}\label{E:HRCentral2}
\Phi_j^c(z_1, \dots, z_n) := \vartheta\left[\dfrac{j-1}{m}, 0\right](mw + \xi\,|\,m\tau),
\end{equation}
where $w := z_1 + \dots + z_n$. Finally, we put:
\begin{equation}\label{E:HRSpace}
\Phi_j := \Phi_j^c \cdot D_m = \vartheta\left[\dfrac{j-1}{m}, 0\right](mw + \xi\,|\,m\tau) \left(\prod\limits_{p < q} \vartheta(z_p - z_q)\right)^m
\end{equation}
and $\mathsf{V}_{m, n, \xi} := 
\left\langle \Phi_1, \dots, \Phi_m \right\rangle_{\CC}$.
Note that the functions $\Phi_j$ are symmetric for $m \in 2 \mathbb{N}$ (bosonic case) and antisymmetric for  $m \in 2 \mathbb{N}_0 +1$ (fermionic case). The space 
$\mathsf{V}_{m,  n, \xi}$ was introduced by Haldane and Rezayi 
 in  \cite{HaldaneRezayi}.  For $m, n \in \NN$ we put:
 $$
 \epsilon_{m, n} := (-1)^{m(n-1)} = 
 \left\{
\begin{array}{cl}
\, 1
& \mbox{\rm if} \, m \in 2 \NN \; \mbox{\rm or} \; m, n \in 2 \NN_0+1  \\
-1
 & \mbox{\rm if} \; m \in 2 \NN_0+1 \;  \mbox{\rm and} \;  n \in 2\NN.
\end{array}
\right. 
 $$ 
\begin{lemma}\label{L:HRwavefunction} Let $X := 
\underbrace{E \times \dots \times E}_{n  \; \mbox{\scriptsize{\sl times}}}$ and 
\begin{equation}
\kW_{m, n, \xi} := \left\{
\begin{array}{cl}
\kL_{mn,_\xi} \boxtimes \dots \boxtimes \kL_{mn, \xi}
& \mbox{\rm if} \; \epsilon_{m, n} = \; 1  \\
\kL_{mn,_\xi}^\sharp \boxtimes \dots \boxtimes \kL^\sharp_{mn, \xi}
 & \mbox{\rm if} \; \epsilon_{m, n} = -1.
\end{array}
\right. 
\end{equation}
Then the following results are  true:
$$
\mathsf{V}_{m,n,  \xi} \subset \Gamma\bigl(X, \kW_{m, n, {\xi}}\bigr) \cong
\left\{
\begin{array}{cl}
\mathsf{W}_{mn, \xi} \otimes \dots \otimes \mathsf{W}_{mn, \xi}
& \mbox{\rm if} \; \epsilon_{m, n} = \; 1  \\
\mathsf{W}_{mn, \xi}^\sharp \otimes \dots \otimes \mathsf{W}_{mn, \xi}^\sharp
 & \mbox{\rm if} \; \epsilon_{m, n} = -1.
\end{array}
\right. 
$$
In particular, $\mathsf{V}_{m, n,  \xi}$ is a subspace of the ground space of the  self-adjoint operator
\begin{equation}\label{E:manybodySchroedinger}
\Delta = \sum\limits_{j = 1}^n \Delta_j^{(mn)} = -\dfrac{1}{4} \sum\limits_{j = 1}^n \bigl(\partial_{u_j}^2 + \partial_{v_j}^2 \bigr) - 2\pi i \dfrac{mn}{t} \sum\limits_{j = 1}^n 
v_j \bar\partial_j
\end{equation} 
acting on  the Hilbert space  of global  square integrable sections of  $\kW_{m, n, {\xi}}$.
\end{lemma}

\begin{proof}
Note that 
$
D_m(z_1 + 1, z_2, \dots, z_n) = \epsilon_{m, n} D_m(z_1, z_2, \dots, z_n),
$
whereas 
$$
D_m(z_1 + \tau, z_2, \dots, z_n) = \epsilon_{m, n} \exp\bigl(-\pi i m(n-1) \tau - 2\pi i mn z_1 + 2\pi i mw\bigr) D_m(z_1, z_2, \dots, z_n).
$$
In a similar way, for any $1 \le j \le m$ we have: $\Phi_j^c(z_1 +1, z_2, \dots, z_n) = \Phi_j^c(z_1, z_2, \dots, z_n)$, whereas 
$
\Phi_j^c(z_1 +\tau, z_2, \dots, z_n) = \varphi(w)^m \exp(-2\pi i \xi) \Phi_j^c(z_1 +\tau, z_2, \dots, z_n).
$ It follows that
$$
\left\{
\begin{array}{lcl}
\Phi_j(z_1 + 1, z_2, \dots, z_n) & = &  \epsilon_{m, n} \Phi_j(z_1, z_2, \dots, z_n) \\
\Phi_j(z_1 + \tau, z_2, \dots, z_n) & = &  \epsilon_{m, n}  \varphi(z_1)^{mn}\cdot \exp(-2\pi i \xi) \Phi_j(z_1, z_2, \dots, z_n). 
\end{array}
\right.
$$
Since $\vartheta(z)$ is an odd function,  we get for any $1 \le k \le n$ the same  transformation rules for the function $\Phi_j$ with respect to  the shifts $z_k \mapsto z_k +1$ and $z_k \mapsto z_k + \tau$, implying the statement. 
\end{proof}

\begin{lemma}\label{L:SectionsPullback}
Consider the map $X \stackrel{\mu}\lar E, (z_1, \dots, z_n) \mapsto z_1 + \dots + z_n$. Then for any vector bundle $\kV$ on $E$,  the canonical morphism
$
\Gamma(E, \kV) \stackrel{\mu^\ast}\lar \Gamma\bigl(X, \mu^\ast(\kV)\bigr)
$
is an isomorphism. 
\end{lemma}
\begin{proof}
Since $\mu$ is a surjective morphism, the induced linear map $\mu^\ast$ is injective. Hence, it is sufficient to show that $\dim_{\CC}\bigl(\Gamma\bigl(X, \mu^\ast(\kV)\bigr)\bigr)  = \dim_{\CC}\bigl(\Gamma(E, \kV)\bigr)$. Since $X$ is an abelian variety of dimension $n$, Serre duality implies that
$$
\bigl(\Hom_X(\kO_X, \mu^\ast \kV)\bigr)^\ast  \cong \Ext^n_X(\mu^\ast\kV , \kO_X) \cong 
\Hom_E(\kV, R\mu_\ast\kO_X[n]),
$$
where the latter $\Hom$--space is taken in the derived category of coherent sheaves of $E$. 
In order to compute the complex $R\mu_\ast\kO_X$, consider the commutative diagram 
$$
\xymatrix{
X \ar[rr]^-{\varrho} \ar[rd]_-\mu & &  X \ar[ld]^-{\pi} \\
& E & 
}
$$
where $\varrho(z_1, z_2, \dots, z_n) = (z_1 + \dots + z_n, z_2, \dots, z_n)$ and 
$\pi(z_1, z_2, \dots, z_n) = z_1$. Since $\varrho$ is an isomorphism, we have:
$
R\varrho_\ast \kO_X \cong \varrho_\ast \kO_X \cong \kO_X.
$
It follows from the K\"unneth formula that 
$$
R\mu_\ast\kO_X \cong R\pi_\ast \left(R \varrho_\ast \kO_X\right) \cong R\pi_\ast \kO_X \cong \kO_E \otimes \mathit{R\Gamma}(\widetilde{X}, \kO_{\widetilde{X}}),$$ where 
$\widetilde{X} := 
\underbrace{E \times \dots \times E}_{n-1  \; \mbox{\scriptsize{\sl times}}}$. Hence, 
$
R\mu_\ast\kO_X \cong \bigoplus\limits_{j = 0}^{n-1} 
\left(\kO_E[-j]\right)^{\oplus \binom{n-1}{j}}.
$
Since $\Hom_E(\kV, \kO_E[j]) = 0$ for $j \ge 2$, applying Serre duality once again,  we get the isomorphisms
$$
\bigl(\Hom_X(\kO_X, \mu^\ast \kV)\bigr)^\ast  \cong  \Hom_E\bigl(\kV, \kO_E[1]\bigr) \cong \bigl(\Hom_E(\kO_E, \kV)\bigr)^\ast. 
$$
Lemma is proven. 
\end{proof}

\smallskip
\noindent
For any $1 \le p < q \le n$, consider the map 
$X \stackrel{\sigma_{pq}}\lar E, (z_1, \dots, z_n) \mapsto z_p - z_q$. Then we have: 
$
\sigma_{pq}^\ast \bigl(\kO_E\bigl([0]\bigr)\bigr) \cong 
\kO_X\bigl(\bigl[\Sigma_{pq}\bigr]\bigr),
$
where
$
\Sigma_{pq} := \bigl\{(z_1, \dots, z_n) \in X\, \big| \, z_p = z_q\bigr\}.
$
In particular, one can view the function $\vartheta^m(z_p - z_q)$ as a global holomorphic section of the line bundle $\sigma_{pq}^\ast \bigl(\kO_E\bigl(m[0]\bigr)\bigr) 
\cong \kO_A\bigl(m\bigr[\Sigma_{pq}\bigr]\bigr)$.

\begin{lemma}\label{L:HRSpace}
We have the following description of the Haldane--Rezayi space:
\begin{equation}\label{E:Axiomatic}
\mathsf{V}_{m, n,  \xi} \cong \Gamma\bigl(X, \kW_{m, n, \xi}\big(- m \bigl[\Sigma\bigr]\bigr)\bigr),
\end{equation}
where $\Sigma = \cup_{p < q} \Sigma_{pq}$.
\end{lemma}

\begin{proof}
Note that for any $m \in \NN$, the function $D_m$ given by 
(\ref{E:HRCentral1}) is a section  of the line bundle
of $\kO_X\bigl(m \bigl[\Sigma\bigr]\bigr) \cong \bigotimes\limits_{p < q} 
\sigma_{pq}^\ast\bigl(\kO_E\bigl(m[0]\bigr)\bigr)
$
and $\mathsf{div}(D_m) = m [\Sigma]$.  We have an isomorphism
$$
\kW_{m, n, \xi} \cong \mu^\ast(\kL_{m, \xi}) \otimes \kO_X \bigl(m \bigl[\Sigma\bigr]\bigr),
$$
hence $\mathsf{V}_{m, \, n, \, \xi} \cong 
\Gamma\bigl(X, \mu^\ast(\kL_{m, \xi})\bigr)$. By Lemma \ref{L:SectionsPullback},
the map  $\Gamma(E, \kL_{m, \xi}) \stackrel{\mu^\ast}\lar 
\Gamma\bigl(X, \mu^\ast(\kL_{m, \xi})\bigr)$ is an isomorphism,  what implies the result. 
\end{proof}

\begin{remark}
In \cite[Definition 2]{Klevtsov}, the space of  \emph{genus one Laughlin's  wave functions} with the filling fraction $\frac{1}{m}$ was defined (following   the original Ansatz of Laughlin \cite{Laughlin}) to be the subspace
$\mathsf{V}(\kL) \subseteq \Gamma(X, \kL \boxtimes \dots \boxtimes \kL)$, where $\kL$ is some ample line bundle on $E$ of degree $mn$ consisting of those sections $\Phi$, which satisfy the following two conditions:
\begin{enumerate}
\item $\Phi$ has zero of order at least $m$ on each partial diagonal $\Sigma_{pq}$ for all $1 \le p < q \le n$. 
\item $\Phi$ is symmetric if $m$ is even, and antisymmetric if $m$ is odd. 
\end{enumerate}
Our considerations show that  it is sufficient to impose only the first condition, whereas the second condition turns out to be \emph{automatically satisfied}. Indeed, the first condition already provides an axiomatic description of the space $\mathsf{V}_{m,n,\xi}$; see (\ref{E:Axiomatic}). On the other hand, $\mathsf{V}_{m,n,\xi} = \left\langle \Phi_1, \dots, \Phi_m\right\rangle_\CC$  with  $\Phi_j$ given by (\ref{E:HRSpace}) for $1\le j \le m$. It is clear all elements of $\mathsf{V}_{m,n,\xi}$ satisfy the second condition.
\end{remark}

\begin{theorem}\label{T:GramMatrixHRSpace} Let $C = \bigl(\langle \Phi_i, \Phi_j\rangle\bigr)_{1 \le i, j \le m}$ be the Gram matrix of the basis $(\Phi_1, \dots, \Phi_m)$ of the Haldane--Rezayi space $\mathsf{V}_{m, n,  \xi}$. Then we have: $C = c(\xi) I_m$ for some  $c(\xi) \in \CC$. 
\end{theorem}

\begin{proof} Let $\mathsf{H}_{m, n, \xi} = L_2(X, \kW_{m, n, \xi})$. By Lemma \ref{L:HRSpace} we have:
\begin{equation}\label{E:HRSpaceCharterized}
\mathsf{V}_{m, n, \xi} = \left\{\Psi \in \mathsf{H}_{m, n, \xi}
\left| \, \begin{array}{l}
\Psi \; \mbox{is holomorphic} \\
\Psi \; \mbox{has vanishing of order at least} \; m \; \mbox{along } \\  \mbox{each diagonal} \;  \Sigma_{pq}, 1 \le p \ne q \le n
\end{array}
\right.
\right\}.
\end{equation}
For a function
$
\Psi(z_1, \dots, z_n) = \psi_1(z_1)\dots  \psi_n(z_n) \in \mathsf{H}_{m, n, \xi}
$
and $k = 1, 2$
we put:
$$
\bigl(T_k(\Psi)\bigr)(z_1, \dots, z_n) = \bigl(T_k(\psi_1)(z_1)\bigr) \dots
\bigl(T_k(\psi_n)(z_n)\bigr), 
$$
where $T_k(\psi_j)$ is given by (\ref{E:MagneticTranslationsFormulae}) for  all $1 \le j \le n$ (regardless of the parity of $\epsilon_{m, n}$). 
This defines a pair of  unitary operators $T_1$ and $T_2$ acting on the Hilbert space $\mathsf{H}_{m, n, \xi}$. For 
 an arbitrary  $\Psi \in \mathsf{H}_{m, n, \xi}$, we have the following explicit formulae: 
\begin{equation*}
\left\{
\begin{array}{l}
\bigl(T_1(\Psi)\bigr)(z_1, \dots, z_n) = \Psi\left(z_1 + \dfrac{1}{mn}, \dots, 
z_n + \dfrac{1}{mn}\right) \\
\bigl(T_2(\Psi)\bigr)(z_1, \dots, z_n) = \exp\Bigl(\dfrac{2\pi i \xi + \pi i \tau}{m}\Bigr) \exp(2\pi iw)\Psi\Bigl(z_1 + \dfrac{\tau}{mn}, \dots, 
z_n + \dfrac{\tau}{mn}\Bigr),
\end{array}
\right.
\end{equation*}
where $w = z_1 + \dots + z_n$. The characterization (\ref{E:HRSpaceCharterized}) allows to conclude that the Haldane--Rezayi subspace $\mathsf{V}_{m, n,  \xi}$ is invariant under the action of $T_1$ and $T_2$. For  
$\mathsf{V}_{m, n,  \xi} \ni \Psi  = \Psi(z_1, \dots, z_n) = \Phi^c(w) D_m(z_1, \dots, z_n)$ we get: 
\begin{equation*}\label{E:ActionCentralFunctions}
\left\{
\begin{array}{l}
\bigl(T_1(\Psi)\bigr)(z_1, \dots, z_n) = \Phi^c\left(w + \dfrac{1}{m}\right) D_m(z_1, \dots, z_n)\\
\bigl(T_2(\Psi)\bigr)(z_1, \dots, z_n) = \exp\Bigl(\dfrac{2\pi i \xi + \pi i \tau}{m}\Bigr) \exp(2\pi iw)
\Phi^c\left(w + \dfrac{\tau}{m}\right) D_m(z_1, \dots, z_n).
\end{array}
\right.
\end{equation*}
It follows that the action of magnetic translations  on the space $\mathsf{V}_{m, n, \xi}$ matches with the analogous action on  $\mathsf{W}_{m, \xi}$, given in Lemma \ref{L:MagneticTranslations}.
According to (\ref{E:HeisenbergRepresentation}) we have: 
\begin{equation*}
T_1\bigl(\Phi_j) = q^{j-1} \Phi_j \; \mbox{and} \; T_2\bigl(\Phi_j) = \Phi_{j+1} \; \mbox{\rm for any} \; 1 \le j \le m,
\end{equation*}
where $\Phi_{m+1} = \Phi_1$. Since $T_1$ and $T_2$ are \emph{unitary} operators, we conclude that
\begin{equation}\label{E:Orthogonality}
\langle \Phi_i, \Phi_j\rangle = 0 \; \mbox{\rm and} \; \langle \Phi_i, \Phi_i\rangle = \langle \Phi_j, \Phi_j\rangle \; 
\mbox{\rm for all} \; 1 \le i \ne j \le m,
\end{equation}
which finishes  the proof. 
\end{proof}

\begin{remark} In the case $\epsilon_{m, n} = 1$, the Heisenberg group $G_{mn}$ acts on the Hilbert space $\mathsf{H}_{m, n, \xi}$ by magnetic translations; see Lemma \ref{L:MagneticTranslations}. If $\epsilon_{m, n} = -1$, then the group $G_{mn}$ has to replaced by its central extension $\widetilde{G}_{mn}$; see Remark \ref{R:MagneticTranslations}.   Regardless of the parity of $\epsilon_{m, n}$, the Haldane--Rezayi subspace $\mathsf{V}_{m, n, \xi}$ remains invariant under these group actions. The obtained representation of  $G_{mn}$ (respectively, of  $\widetilde{G}_{mn}$) on 
 $\mathsf{V}_{m, n, \xi}$ coincides with the representation of 
$G_m$ on $\mathsf{W}_{m, \xi}$ via the surjective group homomorphism  
$G_{mn} \rightarrowdbl  G_m$ (respectively, 
$\widetilde{G}_{mn} \rightarrowdbl  G_m$). 
Note that no sign issues arise in the bosonic case ($m \in 2\NN$), regardless of the number of particles $n$. 
\end{remark}

\section{Multivariate theta functions  and generalized Heisenberg groups}\label{S:CenterMass}
Let $g \in \NN$ and $\Omega \in \Mat_{g \times g}(\CC)$ be a symmetric matrix whose  imaginary part $\mathfrak{Im}(\Omega)$ is positive definite. Analogously to (\ref{E:ThetaWithCharacteristics}) we define for any 
$\vec{z} \in \CC^g$ and $\vec{a}, \vec{b} \in \RR^g$ the series:
\begin{equation}\label{E:MultiThetaWithCharacteristics}
\Theta[\vec{a}, \vec{b}](\vec{z}\,|\, \Omega):= \sum\limits_{\vec{k} \in \mathbbm{Z}^{g}} 
\exp\bigl(\pi i (\vec{k} + \vec{a})^t \Omega (\vec{k} + \vec{a})
+  2\pi i (\vec{k} + \vec{a})^t (\vec{z}+\vec{b})\bigr).
\end{equation}
The following results can be found in \cite[Section II.1]{Mumford}.

\begin{proposition}\label{P:BasicTheta} As a function of $\vec{z}$, the theta-series (\ref{E:MultiThetaWithCharacteristics}) converges absolutely and uniformly on {compact subsets of} $\CC^g$ and satisfies the following quasi-periodic conditions:
\begin{equation}\label{E:MultiThetaFunctionsTransfRules}
\left\{
\begin{array}{lcl}
\Theta[\vec{a}, \vec{b}](\vec{z} + \vec{l} | \Omega ) & = & \exp(2\pi i \vec{a}^t \vec{l})  \Theta[\vec{a}, \vec{b}](\vec{z} | \Omega ) \\

\Theta[\vec{a}, \vec{b}](\vec{z} + \Omega \vec{l} | \Omega )  & = & \exp(- 2\pi i \vec{l}^t (\vec{z} + \vec{b}) - \pi i \vec{l}^t \Omega \vec{l}) \Theta[\vec{a}, \vec{b}](\vec{z} | \Omega )
\end{array}
\right.
\end{equation}
for any $\vec{l} \in \ZZ^g$. 

Next, consider the abelian variety $B := \CC^g/(\ZZ^g + \Omega \ZZ^g)$ and the so-called theta line bundle $\kT_{\vec\xi}$ on $B$ (where  $\vec{\xi} = \Omega\vec{a} + \vec{b}$) defined as $\kT_{\vec\xi} = \CC^g \times \CC/\sim$, where 
\begin{equation*}
 (\vec{z}, v) \sim \bigl(\vec{z} + \vec{l}, \exp(2\pi i \vec{a}^t \vec{l})v\bigr) \sim \bigl(\vec{z} + \Omega \vec{l}, \exp(- 2\pi i \vec{l}^t (\vec{z} + \vec{b}) - \pi i \vec{l}^t \Omega \vec{l})\bigr)
\end{equation*}
for all $\vec{l} \in \ZZ^g$. Then $\kT_{\vec\xi}$ is ample and $\Gamma(B, \kT_{\vec\xi}) = 
\CC \cdot \Theta[\vec{a}, \vec{b}](\vec{z} | \Omega )$. In other words, the vector space of holomorphic functions on $\CC^g$ satisfying the quasi-periodicity constraints 
(\ref{E:MultiThetaFunctionsTransfRules}) has dimension one and is generated by the multi-variate theta-function $\Theta[\vec{a}, \vec{b}](\vec{z} | \Omega)$. 
\end{proposition}

\begin{definition}\label{D:Kmatrix}
For a  symmetric and positive definite matrix $K \in \Mat_{g \times g}(\ZZ)$ we put 
\begin{equation}\label{E:GroupPi}
\Pi := \Pi_K = K^{-1} \ZZ^g/\ZZ^g \cong \ZZ^g/ K \ZZ^g.
\end{equation}
\end{definition}
\begin{remark}
Since the matrix $K$ is non-degenerate, the group  $\Pi$ is finite. Moreover, we have: $|\Pi| = \delta:= \big|\det(K)\big|$. 
\end{remark}

\noindent
Let $\Gamma = \Gamma_\delta = \left\langle \exp\left(\dfrac{2\pi i}{\delta}\right) \right\rangle \subset \CC^\ast$. Then we have a pairing
\begin{equation}\label{E:FormUpsilon}
\Pi \times \Pi \stackrel{\upsilon}\lar \Gamma, \quad \bigl([\vec{a}],  [\vec{b}]\bigr) \mapsto \exp\bigl(2\pi i \vec{a}^t K \vec{b}\bigr).
\end{equation}
Consider the abelian group  $\Pi^{\oplus 2} = \Pi \times \Pi$ and the pairing
\begin{equation}
\Pi^{\oplus 2} \times \Pi^{\oplus 2} \stackrel{\omega}\lar \Gamma, 
\left(\bigl([\vec{a}_1], [\vec{b}_1]\bigr), \bigl([\vec{a}_2], [\vec{b}_2]\bigr)\right) \mapsto \upsilon\bigl([\vec{a}_1], [\vec{b}_2]\bigr).
\end{equation}
It is easy to see that $\omega$ satisfies the $2$-cocycle conditions 
$$
\omega(c_1, c_2) \cdot \omega(c_1 + c_2, c_3) = \omega(c_2, c_3) \cdot \omega(c_1, c_2 + c_3) \quad \mbox{\rm and} \quad \omega(c, 0) = 1 =  \omega(0, c)
$$
for any elements $c_1, c_2, c_3, c \in \Pi^{\oplus 2}$. 
\begin{definition}
The finite Heisenberg group $G_K$ associated with  a symmetric and positive definite matrix $K \in \Mat_{g \times g}(\ZZ)$   is the set $\Pi^{\oplus 2}_K \times \Gamma_\delta$ equipped with the product
$$
(c_1, \gamma_1) \cdot (c_2, \gamma_2) := \bigl(c_1 + c_2, \omega(c_1, c_2) \gamma_1 \gamma_2\bigr).
$$
In particular, $G_K$ is a central extension
$$
0 \lar \Gamma_\delta \lar G_K \lar \Pi^{\oplus 2}_K \lar 1. 
$$
\end{definition}

\smallskip
\noindent
Note that for $K = (k)$, the group $G_K$ coincides with the group $G_k$ introduced in Definition \ref{D:HeisenbergGroup}. Our next goal is to generalize Lemma \ref{L:MagneticTranslations} constructing  a unitary  action of $G_K$ on an appropriate space of multi-variate theta-functions. 

\smallskip
\noindent
We put  $A := 
\underbrace{E \times \dots \times E}_{g  \; \mbox{\scriptsize{\sl times}}}$ and $\Omega = \tau K$.  Obviously,
\begin{equation}
A =   \CC^g/(\ZZ^g + \tau \ZZ^g) \stackrel{\kappa}\lar B := \CC^g/(\ZZ^g + \Omega \ZZ^g), \;  \vec{z} \mapsto K\vec{z}
\end{equation}
is an isogeny of $g$-dimensional abelian varieties, 
whose kernel is equal to $\Pi$. In particular, $\kappa$ has degree  $\delta$. 

\smallskip
\noindent
For any $\vec{l} \in \RR^g$ consider the function 
\begin{equation}
U_{\vec{l}}(\vec{z}, \vec{\xi}) =  \exp\bigl(-\pi i (\vec{l}, 2 \vec{\xi} + 2K \vec{z} + \Omega \vec{l})\bigr).
\end{equation}
It is easy to see that
$$
U_{\vec{l}_1 + \vec{l}_2}(\vec{z}, \vec{\xi}) = U_{\vec{l}_1}(\vec{z} + \tau \vec{l}_2, \vec{\xi})\cdot  U_{\vec{l}_2}(\vec{z}, \vec{\xi}) \quad \mbox{for all} \quad \vec{l}_1,  \vec{l}_2 \in \RR^g.
$$
Moreover, $U_{\vec{l}}(\vec{z} + \vec{n}, \vec{\xi}) = U_{\vec{l}}(\vec{z}, \vec{\xi})$ for any $\vec{n} \in \ZZ^g$ provided $\vec{l} \in K^{-1} \ZZ^g$. 
Now, consider the line bundle on $B$ given by the automorphy factor 
\begin{equation}\label{E:ThetaBundleMod}
 (\vec{z}, v) \sim \bigl(\vec{z} + \vec{m}, v\bigr) \sim \bigl(\vec{z} + \Omega \vec{m}, \exp\bigl(-\pi i (\vec{m}, 2 \vec{\xi} + 2\vec{z} + \Omega \vec{m}) v\bigr)\bigr)
\end{equation}
for all $\vec{m} \in \ZZ^g$. It is easy to see that (\ref{E:ThetaBundleMod}) provides an equivalent description of the line bundle $\kT_{\vec\xi}$ introduced in the proof of  Proposition \ref{P:BasicTheta}. Let
$\kR_{\vec\xi} := \kappa^\ast\bigl(\kT_{\vec\xi}\bigr)$. Then  we have: 
\begin{equation}\label{E:HolomSectionsDegreeg}
\mathsf{W}_{K, \vec\xi} :=  \left\{
\CC^g \stackrel{H}\lar \CC \; \mbox{\rm holomorphic}\left| \, 
\begin{array}{l} 
\begin{array}{l}
H(\vec{z}+ \vec{l}) = H(\vec{z}) \\
H(\vec{z}+ \tau \vec{l}) = 
U_{\vec{l}}(\vec{z}, \vec{\xi})  H(\vec{z}) 
\end{array}
 \mbox{\rm for all} \; \vec{l} \in \ZZ^g\\
\end{array}
\right.
\right\}
\end{equation}
is the space of global holomorphic sections of the line bundle $\kR_{\vec\xi}$.

\smallskip
\noindent
For any $\vec{z} \in \CC^g$ we shall write:
$$
\vec{z} = \vec{x} + \tau \vec{y} = (\vec{x} + s \vec{y}) + it \vec{y},
$$
where $\vec{x}, \vec{y} \in \RR^g$ and $\tau = s + i t$ with $s, t \in \RR$. 
Similarly, for $\vec{\xi} \in \CC^g$ we shall write:
$$
\vec{\xi} = \vec{b} + \tau \vec{a} = (\vec{b} + s \vec{a}) + it \vec{a},
$$
where $\vec{a}, \vec{b} \in \RR^g$. Analogously to (\ref{E:MetricLB1}), consider the function
\begin{equation}\label{E:MetricLBmv}
h(\vec{z}) = h_{K, \vec\xi}\,(\vec{x}, \vec{y}) := 
\exp\bigl(- 2 \pi t (\vec{y}, K \vec{y} + 2 \vec{a})\bigr).
\end{equation}
It is clear that $h$ has the following quasi-periodicity properties: 
$$
h(\vec{x} + \vec{l}, \vec{y}) = h(\vec{x}, \vec{y}) \quad \mbox{\rm and} \quad h(\vec{x}, \vec{y} + \vec{l}) = \exp\bigl(-2 \pi t (\vec{l}, K \vec{l} + 2 K \vec{y} + 2\vec{a})\bigr) h(\vec{x}, \vec{y})
$$
for any $\vec{l} \in \ZZ^g$. 

\begin{lemma} The function $h$ defines a hermitian metric on the line bundle 
$\kR_{\vec{\xi}}$. 
\end{lemma}

\begin{proof} A smooth function $\CC^g \stackrel{\Phi}\lar \CC$ is a smooth section of the line bundle $\kR_{\vec{\xi}}$ if 
$$
\Phi(\vec{x} + \vec{l}, \vec{y}) = \Phi(\vec{x}, \vec{y}) \quad \mbox{\rm and} \quad \Phi(\vec{x}, \vec{y}  + \vec{l}) = U_{\vec{l}}(\vec{z}, \vec{\xi}) \Phi(\vec{x}, \vec{y})
$$
for any $\vec{l} \in \ZZ^g$. Observe that
$
\big|U_{\vec{l}}(\vec{z}, \vec{\xi})\big| = \exp\bigl(\pi t (\vec{l}, 2 \vec{a} + K (2 \vec{y} + \vec{l})\bigr). 
$ Hence, the function
$$
\CC^g \stackrel{\eta_\Phi}\lar \RR_+, \; \vec{z} \mapsto 
\big|\Phi(\vec{z})\big|^2 h(\vec{z})
$$
is periodic:
$
\eta_\Phi(\vec{x} + \vec{l}, \vec{y}) = \eta_\Phi(\vec{x}, \vec{y})  = \eta_\Phi(\vec{x}, \vec{y}  + \vec{l})
$
for any $\vec{l} \in \ZZ^g$. Hence, $\eta_\Phi$ descends to a function 
$A \stackrel{\eta_\Phi}\lar \RR_+$ and $h$ defines a hermitian metric on $\kR_{\vec\xi}$, as asserted.  Note  that  $\eta_{\Phi}(\vec{z})$ is the square of the length of the value of the section $\Phi$ at the point $[\vec{z}] \in A$. 
\end{proof}

\begin{corollary} The vector space $\mathsf{W}_{K, \vec{\xi}}$ gets equipped with a hermitian scalar product given by the formula
\begin{equation}\label{E:scalarproductmv}
\langle \Phi_1, \Phi_2 \rangle := \iint\limits_{[0, 1]^{2g}} 
\exp\bigl(-2\pi  t(\vec{y}, K \vec{y} + 2\vec{a})\bigr)
 \Phi_1(\vec{x}, \vec{y})  \overline{\Phi_2(\vec{x}, \vec{y})}  d\vec{x} \wedge d \vec{y}.
\end{equation}
\end{corollary}

\noindent
For any smooth function $\CC^g \stackrel{\Phi}\lar \CC$ and $\vec{a}, \vec{b} \in \RR^g$ we put: 
\begin{equation}\label{E:MagneticMultivariate}
\left\{
\begin{array}{l c l}
\bigl(S_{\vec{a}}(\Phi)\bigr)(\vec{z})  & = & \Phi(\vec{z} + \vec{a}) \\
\bigl(R_{\vec{b}}(\Phi)\bigr)(\vec{z})  & = & U_{\vec{b}}(\vec{z}, \vec\xi)^{-1} \Phi(\vec{z} + \tau \vec{b}).
\end{array}
\right.
\end{equation}
It is easy to see that for any $\vec{a}_1, \vec{a}_2, \vec{b}_1, \vec{b}_2 \in \RR^g$ we have: 
\begin{equation}\label{E:HG1}
S_{\vec{a}_1} \cdot  S_{\vec{a}_2} = S_{\vec{a}_1 + \vec{a}_2} \quad \mbox{\rm and} \quad   R_{\vec{b}_1} \cdot  R_{\vec{b}_2} = R_{\vec{b}_1 + \vec{b}_2}
\end{equation}
Moreover,
\begin{equation}\label{E:HG2}
S_{\vec{a}} \cdot R_{\vec{b}} = \exp\bigl(2\pi i(\vec{a}, K\vec{b})\bigr) R_{\vec{b}} \cdot S_{\vec{a}} \quad \mbox{\rm for all} \quad  \vec{a}, \vec{b} \in \RR^g.
\end{equation}

\begin{theorem}\label{T:HeisenbergGroupRepr} The following results are true.
\begin{itemize}
\item The Heisenberg group $H_K$ acts on $\mathsf{W}_{K, \vec{\xi}}$ by unitary operators: 
\begin{equation}\label{E:HG3}
H_K \ni \bigl(([\vec{a}], [\vec{b}]), \gamma\bigr) \mapsto 
\gamma R_{\vec{b}} \cdot S_{\vec{a}}.
\end{equation}
Moreover, this representation is irreducible. 
\item The dimension of $\mathsf{W}_{K, \vec{\xi}}$ is $\delta$. Moreover, it has a distinguished basis  $\bigl(H_{\vec{c}}\bigr)_{\vec{c} \in \Pi}$, where 
\begin{equation}\label{E:CentralWaveFunctions}
H_{\vec{c}}(\vec{z}) := \Theta\bigl[\vec{c}, \vec{0}\bigr](K \vec{z} + \vec{\xi} | \Omega).
\end{equation}
\item For any $\vec{a}, \vec{b}, \vec{c} \in \Pi$ we have: 
\begin{equation}\label{E:ActionThetaBasis}
S_{\vec{a}}(H_{\vec{c}})  =  \upsilon(\vec{a}, \vec{c}) H_{\vec{c}} \quad
\mbox{\rm and} \quad 
R_{\vec{b}}(H_{\vec{c}}) =  H_{\vec{c} + \vec{b}}.
\end{equation}
\item For any $\vec{c}_1 \ne \vec{c}_2 \in \Pi$ we have:
\begin{equation}\label{E:ThetaOrthogonal}
\bigl\langle H_{\vec{c}_1}, H_{\vec{c}_2}\bigr\rangle = 0 \quad \mbox{\rm and} \quad 
\bigl\langle H_{\vec{c}_1}, H_{\vec{c}_1}\bigr\rangle = \bigl\langle  H_{\vec{c}_2}, H_{\vec{c}_2}\bigr\rangle.
\end{equation}
\end{itemize}
\end{theorem}

\begin{proof} First note that for any $\vec{a}, \vec{b} \in K^{-1} \ZZ^g$ and $H \in \mathsf{W}_{K, \vec{\xi}}$ we have: 
$S_{\vec{a}}(H), R_{\vec{b}}(H) \in \mathsf{W}_{K, \vec{\xi}}$. Moreover,
$S_{\vec{a}}(H) = H$ (respectively, $R_{\vec{b}}(H) = H$) if $\vec{a} \in \ZZ^g$ (respectively, $\vec{b} \in \ZZ^g$). For a simplicity of notation, we identify 
an element of $K^{-1} \ZZ^g$ with its residue class in the group $\Pi$. It follows that the assignments $\Pi \ni \vec{a} \mapsto S_{\vec{a}}$ and $\Pi \ni \vec{b} \mapsto R_{\vec{b}}$ give representations of the group $\Pi$ on the vector space $\mathsf{W}_{K, \vec{\xi}}$. It follows from (\ref{E:HG2}) that $S_{\vec{a}} \cdot R_{\vec{b}} = \upsilon(\vec{a}, \vec{b}) R_{\vec{b}} \cdot S_{\vec{a}}$ for any $\vec{a}, \vec{b} \in \Pi$.  As a consequence, the assignment (\ref{E:HG3}) indeed provides a representation of the Heisenberg group $G_K$ on the vector space $\mathsf{W}_{K, \vec{\xi}}$.

\smallskip
\noindent
According to (\ref{E:scalarproductmv}), for any $\Phi \in \mathsf{W}_{K, \vec{\xi}}$ we have: 
$$
\big\lVert\Phi \big\rVert^2 = \iint\limits_{[0, 1]^{2g}} 
\exp\bigl(-2\pi  t(\vec{y}, K \vec{y} + 2\vec{a})\bigr)
 \big|\Phi(\vec{x}, \vec{y})\big|^2  d\vec{x} \wedge d \vec{y}.
$$
Hence, for any $\vec{a} \in \Pi$ we get: 
$$
\big\lVert S_{\vec{a}}(\Phi) \big\rVert^2 = \iint\limits_{[0, 1]^{2g}} 
\exp\bigl(-2\pi  t(\vec{y}, K \vec{y} + 2\vec{a})\bigr)
 \big|\Phi(\vec{x} + \vec{a}, \vec{y})\big|^2  d\vec{x} \wedge d \vec{y} = \big\lVert\Phi \big\rVert^2.
$$
An analogous but longer computation shows that
$$
\big\lVert R_{\vec{b}}(\Phi) \big\rVert^2 = \iint\limits_{[0, 1]^{2g}} 
\exp\bigl(-2\pi  t(\vec{y} + \vec{b}, K \vec{y} + K \vec{b} + 2\vec{a})\bigr)
 \big|\Phi(\vec{x}, \vec{y} + \vec{b})\big|^2  d\vec{x} \wedge d \vec{y} = \big\lVert\Phi \big\rVert^2
$$
for any $\vec{b} \in \Pi$. See also the proof of Lemma  \ref{L:MagneticTranslations} for an analogous argument in the case $g = 1$. Thus, we proved  that the constructed representation of $G_K$ on $\mathsf{W}_{K, \vec{\xi}}$ is indeed unitary. Its irreducibility will be shown below.

Recall that $\kR_{\vec{\xi}} = \kappa^\ast\bigl(\kT_{\vec\xi}\bigr)$, where 
$\kT_{\vec\xi}$ is the theta line bundle  introduced in the proof of  Proposition \ref{P:BasicTheta}.
Since $\kT_{\vec\xi}$ is  ample  and the canonical sheaf of $B$ is trivial, the Kodaira vanishing theorem \cite[Section I.1]{GH} implies that $H^p(B, \,  \kT_{\vec{\xi}}) = 0$ for all $p \ge 1$. As  $\kappa$ is an affine  morphism, it follows that  $\kR_{\vec{\xi}} := \kappa^\ast \kT_{\vec{\xi}}$ is an ample line bundle on $A$ as well. Since the canonical sheaf of $A$ is also trivial, the Kodaira vanishing theorem implies again that $H^p\bigl(A,  \kR_{\vec{\xi}}\bigr) = 0$ for  all $p \ge 1$.  By \cite[Corollary 3.6.6]{BirkenhakeLange} we have: $\chi\bigl(A, \kR_{\vec{\xi}}\bigr) = \mathrm{deg}(\kappa) \, \chi\bigl(B, \kT_{\vec{\xi}}\bigr) = \delta$. Hence,  the vector space $\mathsf{W}_{K, \vec\xi} \cong \Gamma\bigl(A,  \kR_{\vec{\xi}}\bigr)$ has dimension $\delta$, as asserted.

It follows from the transformation rules (\ref{E:MultiThetaFunctionsTransfRules}) that $H_{\vec{c}} \in \mathsf{W}_{K, \vec\xi}$ for all $\vec{c} \in \Pi$. Their linear independence will follow from the orthogonality (\ref{E:ThetaOrthogonal}), which will be shown later.

\smallskip
\noindent
Using (\ref{E:MultiThetaFunctionsTransfRules}) we get: 
\begin{equation*}
S_{\vec{a}}\bigl(H_{\vec{c}}\bigr) = \Theta\bigl[\vec{c}, \vec{0}\bigr]\bigl(K\vec{z} + \vec{\xi} + K \vec{a} \, \big|\, \Omega\bigr) = \exp\bigl(2\pi i (\vec{c}, K\vec{a})\bigr) \Theta\bigl[\vec{c}, \vec{0}\bigr]\bigl(K\vec{z} + \vec{\xi} \, \big|\, \Omega\bigr) = \upsilon(\vec{a}, \vec{c})  H_{\vec{c}}.
\end{equation*}
Using the formula
$$
\Theta\bigl[\vec{c}, \vec{0}\bigr]\bigl(\vec{z}\, \big|\, \Omega\bigr) = 
\exp\bigl(\pi i (\vec{c}, \Omega\vec{c}) + 2\pi i (\vec{c}, \vec{z})\bigr) \Theta\bigl(\vec{z} + \Omega\vec{c}\, \big|\, \Omega\bigr)
$$
we also deduce that  $R_{\vec{b}}\bigl(H_{\vec{c}}\bigr) = H_{\vec{c} + \vec{b}}$. Since $R_{\vec{b}}$ is a unitary operator, we conclude that
$\bigl\langle H_{\vec{c}_1}, H_{\vec{c}_1}\bigr\rangle = \bigl\langle H_{\vec{c}_2}, H_{\vec{c}_2}\bigr\rangle$ for all $\vec{c}_1, \vec{c}_2 \in \Pi$. 

For any $1 \le j \le g$ let $\vec{e}_j \in \ZZ^g$ be the corresponding element of the canonical basis and $\vec{u}_j = K^{-1} \vec{e}_j$. Using unitarity of $S_{\vec{u}_j}$ and the formula for its action on $H_{\vec{c}}$ we get:
$$
\bigl\langle H_{\vec{c}_1}, H_{\vec{c}_2}\bigr\rangle  = 
\bigl\langle S_{\vec{u}_j}(H_{\vec{c}_1}), S_{\vec{u}_j}(H_{\vec{c}_2})\bigr\rangle = \upsilon(\vec{c}_1 - \vec{c}_2, \vec{u}_j) \bigl\langle H_{\vec{c}_1}, H_{\vec{c}_2}\bigr\rangle.
$$
If $\vec{c}_1 \ne \vec{c}_2$ then there exists $1 \le j \le g$ such that 
$\upsilon(\vec{c}_1 - \vec{c}_2, \vec{u}_j) \ne 1$. Hence, $\bigl(H_{\vec{c}_1}, H_{\vec{c}_2}\bigr) = 0$, as asserted. It follows that the family of elements 
$\bigl(H_{\vec{c}}\bigr)_{\vec{c} \in \Pi}$ is linearly independent. Since the dimension of $\mathsf{W}_{K, \vec{\xi}}$ is already known to be 
$\delta = \big| \Pi\big|$, this shows that $\bigl(H_{\vec{c}}\bigr)_{\vec{c} \in \Pi}$ is indeed a basis of this space. 

It remains to prove that the action of the Heisenberg group $G_K$ on $\mathsf{W}_{K, \vec{\xi}}$ is irreducible. First note that for any $\vec{b} \in \Pi$ we have a group homomorphism
$$
\Pi \stackrel{X_{\vec{b}}}\lar \CC^\ast, \; \vec{a} \mapsto \upsilon(\vec{a}, \vec{b}) = \exp\bigl(2\pi i(\vec{a}, K\vec{b})\bigr),
$$
i.e.~$X_{\vec{b}}$ is a one-dimensional representation of $\Pi$. It is easy to see that $X_{\vec{b'}} \cong X_{\vec{b}''}$ if and only if $\vec{b'} = \vec{b''}$. Since  $\Pi$ is a finite abelian group, all its irreducible representations are one dimensional and there are precisely $\delta$ of them (up to an isomorphism). Next, $\Pi \ni \vec{a} \stackrel{S}\mapsto S_{\vec{a}}$ gives a representation of $\Pi$ on the vector space $\mathsf{W}_{K, \vec{\xi}}$. It follows from (\ref{E:ActionThetaBasis}) that this representations splits into a direct sum: 
$
S \cong \bigoplus\limits_{\vec{a} \in \Pi} X_{\vec{a}}.
$
From the corresponding character formula it follows that
$
\delta = \frac{1}{\delta} \sum\limits_{\vec{a} \in \Pi} \big| \mathsf{tr}(S_{\vec{a}})\big|^2.
$
Now, let $\chi$ be the character of the representation of $G_K$ on $\mathsf{W}_{K, \vec{\xi}}$. Taking into account the formulae (\ref{E:ActionThetaBasis}) we see that the trace of the actions of  $\bigl((\vec{a}, \vec{b}), \gamma\bigr)$ is zero provided $\vec{b} \ne \vec{0} \in \Pi$. Now, let $H_K \stackrel{\chi}\lar \CC$ be the character of the representation of $G_K$ on $\mathsf{W}_{K, \vec{\xi}}$. Since $\big|G_K\big| = \delta^3$, we get: 
$$
(\chi, \chi) = \frac{1}{\delta^3} \sum\limits_{\gamma \in \Gamma}\sum\limits_{\vec{a} \in \Pi} \big| \mathsf{tr}(\gamma S_{\vec{a}})\big|^2 = 1,
$$
where the fact that $|\gamma| = 1$ for all $\gamma \in \Gamma$ was used. This shows that the action of $G_K$ on $\mathsf{W}_{K, \vec{\xi}}$ is irreducible, as asserted. 
 \end{proof}

\begin{remark} One can prove by a straightforward computation that 
\begin{equation}\bigl\langle H_{\vec{c}_1}, H_{\vec{c}_2}\bigr\rangle= \kappa(\vec\xi) \cdot \delta_{\vec{c}_1,\vec{c}_2},
\end{equation} 
where $\kappa(\vec\xi)= \left(2t\delta\right)^{-\frac{g}{2}}\cdot 
\exp\bigl(2\pi t (\vec{a},K^{-1} \vec{a})\bigr)$; see \cite[Proposition 5.1]{BurbanKlevtsov}. 
\end{remark}

\section{Multi-layer model of FQHE on a torus}\label{S:WaveFunctions} 

\noindent
Following \cite{KeskiVakkuriWen, Wen}, 
the multi-layer torus model of FQHE is defined by the following datum.
 
\begin{definition}[Wen datum]\label{D:WenData} 
Let  $K \in \Mat_{g \times g}(\NN_0)$ be a matrix satisfying the following conditions.
\begin{itemize}
\item $K$ is symmetric and positive definite.
\item All diagonal entries of $K$ are either even ($\epsilon(K) := 1$,  bosonic case) or odd ($\epsilon(K) := -1$, fermionic case). 
\item All entries of the vector $\vec{u} := K^{-1} \vec{e} \in \mathbb{Q}^g$
are positive, where $\vec{e} = \left(\begin{array}{c} 1 \\ \vdots \\ 1\end{array}\right)$. 
\end{itemize}
 Next, let  $\vec\nn = \left(\begin{array}{c} n_1 \\ \vdots \\ n_g\end{array}\right) \in \NN^g$ be such that $K \vec\nn = d \vec{e}$ for some $d \in \NN$. 
Following \cite{Wen} (see also \cite{WenZee, KeskiVakkuriWen}) we call such a  pair $(K, \vec\nn)$ a \emph{Wen datum} and $K$ a \emph{Wen matrix}. We put:
$n := n_1 + \dots + n_g$, $\delta := \det(K)$ and $\rho:=  \dfrac{n\delta}{d}$. 
\end{definition}

\begin{lemma}\label{L:WenDatumCyclic} For any Wen datum $(K, \vec\nn)$ we have: $\vec{u} \in K^{-1} \ZZ^g$ and $\rho \in \ZZ$. Moreover, if $\mathsf{gcd}\bigl(\delta,  \rho\bigr) = 1$ then the class of $\vec{u}$ generates the group $\Pi = \Pi_K$.
\end{lemma}

\begin{proof} 
By the definition of $\vec{u}$ we have: 
 $$\vec{u}  = \frac{1}{d}\left(\begin{array}{c} n_1 \\ \vdots \\ n_g\end{array}\right) = 
K^{-1} \vec{e} = \frac{1}{\delta} K^\sharp \vec{e},
$$
where $K^\sharp$ is the adjunct matrix of $K$ (i.e. $K K^\sharp = \delta I$, where $I$ is the identity matrix). It is clear that $\vec{u} \in K^{-1} \ZZ^g$.
Since $K^\sharp \in \Mat_{g \times g}(\ZZ)$, we conclude  that  $\delta \vec{u} \in \ZZ^g$. The sum of all entries of $\delta \vec{u}$ is an integer, hence  
$\frac{n\delta}{d} \in \ZZ$, as asserted.

Suppose now that $\mathsf{gcd}\bigl(\delta,  \rho\bigr) = 1$. Since $|\Pi| = \delta$, the order of $[\vec{u}] \in \Pi$ is a divisor of $\delta$. If $\delta = s \bar{\delta}$ with $s > 1$ and $\bar\delta \vec{u} \in \ZZ^g$ then $\frac{n_i \bar\delta}{d} \in \ZZ$ for all $1\le i \le g$. It follows that 
$\frac{n\bar\delta}{d} \in \ZZ$ and $s \big| \mathsf{gcd}\bigl(\delta,  \rho\bigr)$, yielding a contradiction. 
\end{proof}

\begin{remark} The identity $\vec{n} = \dfrac{d}{\delta} K^\sharp \vec{e}$ implies that
\begin{equation}\label{E:SumAdjunct}
\sum_{i, j =1}^g K_{ij}^\sharp =  \frac{n\delta}{d} = \rho.
\end{equation}
In particular, $\rho$ is determined by Wen's matrix $K$ itself. We shall say that $K$ is \emph{primary} if $\mathsf{gcd}\bigl(\delta,  \rho\bigr) = 1$.
\end{remark}

\begin{example}\label{E:JainWen} For any $p, g \in \NN$ consider the following matrix
\begin{equation}\label{E:WenFroehlich}
K = K_{p, g} = 
\left(
\begin{array}{cccc}
p+1 & p & \dots & p \\
p & p+1 & \dots & p \\
\vdots & \vdots & \ddots & \vdots \\
p & p & \ddots & p+1
\end{array}
\right)
\in \mathsf{Mat}_{g \times g}(\NN).
\end{equation}
Then $K$ has precisely two eigenvalues: $1$ (with multiplicity $g-1$) and $\delta = pg +1$ (the corresponding eigenvector is $\vec{e}$). Hence, $K$ is positive definite and its determinant is $\delta$. We can take 
$\vec{\nn}= m \vec{e}$  for any $m \in \NN$ (hence, $n = mg$). It follows that  $d= m \delta$ and 
$\rho = \frac{n \delta}{d} = g$ is coprime to $pg +1$. Hence, $K$ is a primary Wen's matrix.

Let $N = \left(\begin{array}{ccc} 
1 & \dots & 1 \\
\vdots & \ddots & \vdots \\
1 & \dots & 1
\end{array}\right)$. Then we have: $K = I + p N$. Since $N^2 = g N$, we get: 
$(I + pN)(\delta I - pN) = \delta I$. We get the following expression for the adjunct matrix $K^\sharp$: 
\begin{equation}\label{E:Kadjunct}
K^\sharp = 
\left(
\begin{array}{cccc}
\delta - p  & -p & \dots & -p \\
-p &  \delta- p & \dots & -p \\
\vdots & \vdots & \ddots & \vdots \\
-p & -p & \dots & \delta - p
\end{array}
\right).
\end{equation}
In particular, we see that  $\sum\limits_{i, j = 1}^g K^\sharp_{ij} = g$, which of course matches with the formula (\ref{E:SumAdjunct}).

The matrix $K_{p, g}$ appeared in the context of multi-layer models in a work of Wen and Zee \cite{WenZee}. Shortly afterwards, Fr\"ohlich and Zee pointed out in \cite[Section 5]{FZ} that   $K_{p, g}$ is connected with the root system of type $A_{p-1}$. It would be very interesting to establish a closer relation between the  multi-layer torus models of Keski-Vakkuri and Wen  which we shall introduce below with the approach of Fr\"ohlich et al. \cite{Froehlich} to the FQHE.
\end{example}

\begin{definition}\label{D:WenDatum}
Let $E$ be a complex torus and $(K, \vec\nn)$ be a Wen datum. Then we put:

\noindent
\begin{itemize}
\item $X = \underbrace{E \times \dots \times E}_{n_1  \, \mbox{\scriptsize{\sl times}}} \times \dots \times \underbrace{E \times \dots \times E}_{n_g  \, \mbox{\scriptsize{\sl times}}}$ and $n = n_1 + \dots + n_g = \mathsf{dim}(X)$. 
\item $D_{K, \vec\nn} = 
\prod\limits_{k = 1}^g \left(\prod\limits_{1 \le p < q \le n_k} 
\vartheta\left(z_p^{(k)} - z_q^{(k)} \right)\right)^{K_{kk}} \cdot \prod\limits_{ 1 \le k < l \le g} \left(\prod\limits_{p = 1}^{n_k} \prod\limits_{q = 1}^{n_l}
\vartheta\left(z_p^{(k)} - z_q^{(l)} \right)\right)^{K_{kl}},
$
where $\left(z_1^{(1)}, \dots, z_{n_1}^{(1)}; \dots; z_1^{(g)}, \dots, z_{n_g}^{(g)}\right)$ are the standard local coordinates on the complex torus $X$. 
\item For any $1 \le k \le g$ we put: $w_k = z_1^{(k)} + \dots + z_{n_k}^{(k)}$ and  $\vec{w} = (w_1, \dots, w_g)$. 
\item For any $\vec{c} \in \Pi$  and $\vec\xi \in \CC^g$, the  corresponding wave function of Keski-Vakkuri and Wen \cite{KeskiVakkuriWen} is given by the expression 
\begin{equation}\label{E:KV-Wen-wavefunct}
\Phi_{\vec{c}}\left(z_1^{(1)}, \dots, z_{n_1}^{(1)}; \dots;  z_1^{(g)}, \dots, z_{n_g}^{(g)}\right) =  \Theta\bigl[\vec{c}, \vec{0}\bigr](K \vec{w} + \vec{\xi} | \Omega) \cdot D_{K, \vec\nn}.
\end{equation}
\end{itemize}
\end{definition}

\begin{remark} In the case $g =1$, the functions (\ref{E:KV-Wen-wavefunct}) coincide with wave functions (\ref{E:HRSpace}) of Haldane and Rezayi. For $g > 1$, $1 \le k \le g$ and $\vec{c} \in \Pi$, the wave function $\Phi_{\vec{c}}$ is anti-symmetric for $\epsilon(K) = -1$ and symmetric for $\epsilon(K) = 1$ under the permutations of  $\bigl(z_1^{(k)}, \dots, z^{(k)}_{n_k}\bigr)$. However, they are neither symmetric nor anti-symmetric with respect to the arbitrary permutations of  the variables $\left(z_1^{(1)}, \dots, z^{(1)}_{n_1}, \dots, z_1^{(g)}, \dots, z^{(g)}_{n_g}\right)$. The idea to use non-antisymmetric wave functions with such type of symmetry  in the framework of the FQHE goes back to a work of Halperin \cite{Halperin}. Ancestors of the wave functions  (\ref{E:KV-Wen-wavefunct}) for the FQHE on a plane or on  a sphere were first introduced by Wen and Zee \cite{WenZee}.
\end{remark}

\noindent
We have the following straightforward generalization of Lemma \ref{L:HRwavefunction}.

\begin{lemma}\label{L:KWwavefunction}  Consider the following line bundle $\kW_{K, \vec{\nn}, \vec{\xi}}$ on $X$:
\begin{equation}\label{E:linebundleNbody}
\kW_{K, \vec{\nn}, \vec{\xi}} := \left\{
\begin{array}{cl}
\underbrace{\kL_{d, \xi_1} \boxtimes \dots \kL_{d, \xi_1}}_{n_1  \; \mbox{\scriptsize{\sl times}}} \boxtimes \dots \boxtimes \underbrace{\kL_{d, \xi_g} \boxtimes \dots \boxtimes \kL_{d, \xi_g}}_{n_g  \; \mbox{\scriptsize{\sl times}}} 
& \mbox{\rm if} \, \epsilon(K) + d \in 2 \NN \\
\underbrace{\kL^\sharp_{d, \xi_1} \boxtimes \dots \kL^\sharp_{d, \xi_1}}_{n_1  \; \mbox{\scriptsize{\sl times}}} \boxtimes \dots \boxtimes \underbrace{\kL^\sharp_{d, \xi_g} \boxtimes \dots \boxtimes \kL^\sharp_{d, \xi_g}}_{n_g  \; \mbox{\scriptsize{\sl times}}} & \mbox{\rm if} \, \epsilon(K) + d \in 2 \NN +1.
\end{array}
\right. 
\end{equation}
Then  the wave functions $\bigl(\Phi_{\vec{c}}\bigr)_{\vec{c} \in \Pi}$ of Keski-Vakkuri and Wen belong to the ground state of the analogue of the many-body magnetic Schr\"odinger operator (\ref{E:manybodySchroedinger}).
\end{lemma}

\begin{proof} Let $\epsilon := (-1)^{\epsilon(K) + d}$. A straightforward computation shows that for any $\vec{c} \in \Pi$, $1 \le k \le g$ and $1 \le p \le n_k$ we have:
$$
\Phi_{\vec{c}}\left(z_1^{(1)}, \dots, z_{p}^{(k)} + 1,  \dots, z_{n_g}^{(g)}\right) = \epsilon \Phi_{\vec{c}}\left(z_1^{(1)}, \dots,  z_{p}^{(k)},  \dots, z_{n_g}^{(g)}\right)
$$
and 
$$
\Phi_{\vec{c}}\left(z_1^{(1)}, \dots, z_{p}^{(k)} + \tau,  \dots, z_{n_g}^{(g)}\right) = \epsilon  \exp(\-2\pi i \xi_{k}) 
\, \varphi(z_{p}^{(k)})^d \, \Phi_{\vec{c}}\left(z_1^{(1)}, \dots,  z_{p}^{(k)},  \dots, z_{n_g}^{(g)}\right),
$$
implying the result. 
\end{proof}

\smallskip
\noindent
We need the following slight generalization of Lemma \ref{L:SectionsPullback}.

\begin{lemma}\label{L:sectionsofpullback2} Consider the following morphism of abelian varieties:
$$
X \stackrel{\mu}\lar A, \bigl(z_1^{(1)}, \dots, z_{n_1}^{(1)};  \dots; z_1^{(g)}, \dots, z_{n_g}^{(g)}\bigr) \mapsto \bigl(z_1^{(1)}+  \dots + z_{n_1}^{(1)},   \dots,  z_1^{(g)}+ \dots + z_{n_g}^{(g)}\bigr).
$$
Then for any vector bundle $\kV$ on $A$, the  map $\Gamma(A, \kV) \stackrel{\mu^\ast}\lar \Gamma\bigl(X, \mu^\ast \kV\bigr)$ is an isomorphism. 
\end{lemma}

\begin{proof}
Since the morphism  $\mu$ is surjective, the induced linear map $\mu^\ast$ is injective.  Let $n = n_1 + \dots + n_g$. Applying the Serre duality and adjunction we get: 
$$
\bigl(\Gamma(X, \mu^\ast \kV)\bigr)^\ast  \cong \Hom_X(\mu^\ast\kV , \kO_X[n]) \cong 
\Hom_A(\kV, R\mu_\ast\kO_X[n]).
$$
Analogously to the proof of Lemma \ref{L:SectionsPullback} one can show that
$
R\mu_\ast\kO_X \cong \bigoplus\limits_{j = 0}^{n-g} 
\left(\kO_A[-j]\right)^{\oplus \binom{n-g}{j}}.
$
Since $\Hom_A(\kV, \kO_E[p]) = 0$ for $p > g$, we conclude that
$$
\Hom_A(\kV, R\mu_\ast\kO_X[n]) \cong \bigoplus\limits_{j = 0}^{n-g} \left(\Hom_A(\kV, \kO_A[n-j])\right)^{\oplus \binom{n-g}{j}} \cong \Hom_A(\kV, \kO_A[g]) \cong \bigl(\Gamma(A, \kV)\bigr)^\ast. 
$$
It follows that $\Gamma(A, \kV) \stackrel{\mu^\ast}\lar \Gamma\bigl(X, \mu^\ast \kV\bigr)$ is an injective linear map of vector spaces of equal dimension, hence it is an isomorphism. 
\end{proof}

\begin{definition} Let $\bigl(E, (K, \vec\nn)\bigr)$ be a datum as in Definition \ref{D:WenDatum} (\emph{KV-W datum} for short).
\begin{itemize}
\item For any $1 \le k \le g$ and $1 \le p < q \le n_k$ we put:
$$
\Xi^{(k)}_{p, q} := \left\{\bigl(z_1^{(1)}, \dots, z_{n_1}^{(1)};  \dots; z_1^{(g)}, \dots, z_{n_g}^{(g)}\bigr) \in X \, \big| \, z_p^{(k)} = z_q^{(k)} \right\}.
$$
\item Similarly, for any $1 \le k < l \le g$, $1 \le p \le n_k$  and $1 \le q \le n_l$ we put:
$$
\Xi_{p, q}^{(k, l)} := \left\{\bigl(z_1^{(1)}, \dots, z_{n_1}^{(1)};  \dots; z_1^{(g)}, \dots, z_{n_g}^{(g)}\bigr) \in X \, \big| \, z_p^{(k)} = z_q^{(l)} \right\}.
$$
\item The matrix $K$ from Definition \ref{D:WenData} determines  the divisor 
\begin{equation}\label{E:WenDivisor}
\Xi_K:= \sum\limits_{k = 1}^g K_{kk}\sum\limits_{1 \le p < q \le n_k} \left[\Xi_{p, q}^{(k)}\right] + 
\sum\limits_{1 \le k < l \le g} K_{kl} \sum\limits_{p = 1}^{n_k} \sum\limits_{q = 1}^{n_l} \left[\Xi_{p, q}^{(k, l)}\right]
\end{equation}
on the abelian variety $X$. 
\end{itemize}
We define the  space of wave functions of Keski-Vakkuri and Wen by the formula
\begin{equation}\label{E:SpaceKV-WenWaveFunctions}
\mathsf{V}_{K, \vec{\nn}, \vec{\xi}} := \Gamma\bigl(X, \kW_{K, \vec{\nn}, \vec{\xi}}(-[\Xi_K])\bigr).
\end{equation}
\end{definition}

\begin{remark}\label{R:CenterMassAllWave}
By Lemma \ref{L:KWwavefunction} we have: $\kW_{K, \vec{\nn}, \vec{\xi}}\,(-[\Xi_K])\cong \mu^\ast \kR_{\vec\xi}$.
Applying Lemma \ref{L:sectionsofpullback2}, we obtain: 
\begin{equation}\label{E:CentralWaveIdentify}
\mathsf{V}_{K, \vec{\nn}, \vec{\xi}} \cong \Gamma\bigl(X, \mu^\ast(\kR_{\vec\xi})\bigr) \cong \Gamma\bigl(A, \; \kR_{\vec\xi}\bigr) = \mathsf{W}_{K, \vec{\xi}}, 
\end{equation}
see (\ref{E:HolomSectionsDegreeg}) for a description of $\mathsf{W}_{K, \vec{\xi}}$.
We put: $\mathsf{W}_{K, \vec{\nn}, \vec{\xi}} := \Gamma(X, \kW_{K, \vec{\nn}, \vec{\xi}})$. By the K\"unneth formula we have: 
\begin{equation}\label{E:FullSpaceOfWaveFunctions}
\mathsf{W}_{K, \vec{\nn}, \vec{\xi}} =
\left\{
\begin{array}{cl}
\underbrace{\mathsf{W}_{d, \xi_1} \otimes \dots \otimes \mathsf{W}_{d, \xi_1}}_{n_1  \; \mbox{\scriptsize{\sl times}}} \otimes \dots \otimes \underbrace{\mathsf{W}_{d, \xi_g} \otimes \dots \otimes \mathsf{W}_{d, \xi_g}}_{n_g  \; \mbox{\scriptsize{\sl times}}} 
&  \, \mbox{\rm if} \; \;   \epsilon(K) + d \in 2 \NN \\
\underbrace{\mathsf{W}^\sharp_{d, \xi_1} \otimes  \dots \otimes \mathsf{W}^\sharp_{d, \xi_1}}_{n_1  \; \mbox{\scriptsize{\sl times}}} \otimes  \dots \otimes  \underbrace{\mathsf{W}^\sharp_{d, \xi_g} \otimes  \dots \otimes \mathsf{W}^\sharp_{d, \xi_g}}_{n_g  \; \mbox{\scriptsize{\sl times}}} &  \, \mbox{\rm otherwise}.
\end{array}
\right. 
\end{equation}
As a consequence, we obtain the following concrete description of the vector space of Keski-Vakkuri--Wen wave functions:
\begin{equation}
\mathsf{V}_{K, \vec{\nn}, \vec{\xi}} = \bigl\langle 
\Phi_{\vec{c}} \, | \;  \vec{c} \in \Pi \bigr\rangle_{\CC} \subset \mathsf{W}_{K, \vec{\nn}, \vec{\xi}}
\end{equation}
where $\Phi_{\vec{c}}$ is the function given by (\ref{E:KV-Wen-wavefunct}). 
In what follows,  $\mathsf{V}_{K, \vec{\nn}, \vec{\xi}}$ will be  endowed with the \emph{restricted  hermitian scalar product} of the  ambient  space  $\mathsf{W}_{K, \vec{\nn}, \vec{\xi}}$.
\end{remark}

\begin{remark}\label{R:Matching} Depending on the parity of $\epsilon(K) + d$, we have an unitary action of the Heisenberg group $G_d$ (respectively, of its central extension 
$\widetilde{G}_d$) on the space $\mathsf{W}_{K, \vec{\nn}, \vec{\xi}}$, given by the magnetic translations (\ref{E:MagneticTranslationsFormulae}). Note that the formulae for the actions of $T_1$ and $T_2$ do not depend on the parity of $\epsilon(K) + d$. Moreover, it follows from (\ref{E:SpaceKV-WenWaveFunctions}) that 
$\mathsf{V}_{K, \vec{\nn}, \vec{\xi}}$ is an invariant subspace of this group action. 
\end{remark}

\begin{lemma}\label{L:Magnetic} Let $(E, (K, \vec\nn))$ be an arbitrary KV-W datum and $\vec\xi \in \CC^g$. Then the following diagrams are commutative:
\begin{equation}\label{E:TwoMagnetic}
\begin{array}{ccc}
\xymatrix{
\mathsf{W}_{K, \vec{\xi}} \ar[r]^-{\iota} \ar[d]_-{S_{\vec{u}}}&  \mathsf{V}_{K, \vec\nn, \vec\xi} \ar[d]^-{T_1}\\
\mathsf{W}_{K, \vec{\xi}} \ar[r]^-{\iota} &  \mathsf{V}_{K, \vec\nn, \vec\xi}
}
& \mbox{\rm \begin{tabular}{c} \\  \\ \\ and \\ \end{tabular}} &
\xymatrix{
\mathsf{W}_{K, \vec{\xi}} \ar[r]^-{\iota} \ar[d]_-{R_{\vec{u}}}&  \mathsf{V}_{K, \vec\nn, \vec\xi} \ar[d]^-{T_2}\\
\mathsf{W}_{K, \vec{\xi}} \ar[r]^-{\iota} &  \mathsf{V}_{K, \vec\nn, \vec\xi}
}
\end{array}
\end{equation}
where $\vec{u} = K^{-1}\vec{e}$ and  $\mathsf{W}_{K, \vec{\xi}} \stackrel{\iota}\lar \mathsf{V}_{K, \vec\nn, \vec\xi}$ is the isomorphism of vector spaces  given by the rule 
$\iota(\Phi) = \Phi \cdot D_{K, \vec\nn}$. In particular, we have the following formulae for the action of the magnetic translations $T_1$ and $T_2$ on the basis $\bigl(\Phi_{\vec{c}}\bigr)_{\vec{c} \in \Pi}$ of the vector space 
$\mathsf{V}_{K, \vec\nn, \vec\xi}$:
\begin{equation}\label{E:MagneticBasis}
T_1(\Phi_{\vec{c}}) = \upsilon(\vec{u}, \vec{c}) \Phi_{\vec{c}} \quad \mbox{\rm and} \quad T_2(\Phi_{\vec{c}}) = \Phi_{\vec{c} + \vec{u}},
\end{equation}
where $\upsilon$ is the form given by (\ref{E:FormUpsilon}). 
\end{lemma}
\begin{proof} First note that $\iota$ is just an explicit form of the isomorphism (\ref{E:CentralWaveIdentify}). For  $\Phi \in \mathsf{W}_{K, \vec{\xi}}$, the induced  action of $T_1$ and $T_2$ is given by the following formulae:
\begin{equation*}
\left\{
\begin{array}{l}
\bigl(T_1(\Phi)\bigr)\bigl(\vec{w}\bigr) = \Phi\bigl(\vec{w} + \vec{u}\bigr) \\
\bigl(T_2(\Phi)\bigr)\bigl(\vec{w}\bigr) = 
\exp\Bigl( 2\pi i(\vec{u}, \vec{\xi}) + \pi i \tau \dfrac{n}{d}\Bigr)\exp\bigl(2\pi i n (\vec{e}, \vec{w})\bigr) \Phi\bigl(\vec{w} + \tau\vec{u}\bigr).
\end{array}
\right.
\end{equation*}
A comparison with the formulae (\ref{E:MagneticMultivariate}) shows that 
the both  diagrams (\ref{E:TwoMagnetic}) are commutative, as asserted. The formulae (\ref{E:MagneticBasis}) are consequences of (\ref{E:ActionThetaBasis}). 
\end{proof}

\begin{remark}
Lemma \ref{L:MagneticTranslations} implies that the action of $G_d$ (respectively, $\widetilde{G}_d$) on $\mathsf{W}_{K, \vec{\xi}}$ coincides with the action 
on $\mathsf{V}_{K, \vec\nn, \vec\xi}$ of the subgroup of the Heisenberg group $G_K$ generated by the elements
$\bigl(([\vec{u}], [\vec{0}]), 1\bigr)$ and $\bigl([\vec{0}], [\vec{u}]), 1\bigr)$. In the case when $K$ is primary, the group  $\Pi$ is cyclically generated by the element $[\vec{u}]$ (see Lemma \ref{L:WenDatumCyclic}) and this subgroup coincides with 
the whole group $G_K$.
\end{remark}

\smallskip
\noindent
Now we prove the following generalization of Theorem \ref{T:GramMatrixHRSpace}.

\begin{theorem}\label{T:GramMatrixKVWSpace} Let $\bigl(E, (K, \vec\nn)\bigr)$ be a KV-W datum with a primary matrix $K$. Consider the Gram matrix  $C = \bigl(\langle \Phi_i, \Phi_j\rangle\bigr)_{1 \le i, j \le \delta}$  of the basis $(\Phi_1, \dots, \Phi_\delta)$ of the hermitian vector space $\mathsf{V}_{K, \vec{\nn}, \vec{\xi}}$. Then we have: $C =  c I_\delta$ for some  $c = c(K, \vec{n}, \vec{\xi}) \in \CC$. 
\end{theorem}

\begin{proof}  For any $1 \le i \le \delta$ we put $\Phi_i := \Phi_{i \vec{u}}$. Since $[\vec{u}]$ generates the group $\Pi$, $(\Phi_1, \dots, \Phi_\delta)$ is a basis of the vector space $\mathsf{V}_{K, \vec{\nn}, \vec{\xi}}$. The linear operators $T_1, T_2: \mathsf{W}_{K, \vec\nn,  \vec\xi} \lar \mathsf{W}_{K, \vec\nn, \vec\xi}$ are given in  by the  matrices:
\begin{equation*}
\bigl[T_1\bigr] = \left(
\begin{array}{cccc}
1 & 0 & \dots & 0 \\
0 & q & \dots & 0 \\
\vdots & \vdots& \ddots & \vdots \\
0 & 0 & \dots & q^{\delta-1}
\end{array}
\right) \quad \mbox{\rm and} \quad
\bigl[T_2\bigr] = \left(
\begin{array}{cccc}
0 & \dots & 0 & 1 \\
1 & \dots & 0 & 0 \\
\vdots & \ddots& \vdots & \vdots \\
0 & \dots & 1 & 0
\end{array}
\right),
\end{equation*}
where $q = \exp\left(2\pi i \dfrac{n}{d}\right)$. Since $\dfrac{n}{d} = \dfrac{\rho}{\delta}$  and $\gcd(\rho, \delta) = 1$, we conclude that $q$ is a primitive root of $1$ of order $\delta$. 
The stated result about the Gram matrix $C$ follows from the fact that $T_1$ and $T_2$ are unitary operators.
\end{proof}

\section{Vector bundles and Fourier--Mukai transforms on abelian varieties}\label{S:FMT}
Our goal is to extend the correspondence $\CC^q \ni  \vec{\xi} \mapsto \mathsf{V}_{K, \vec{\nn}, \vec{\xi}}$ to a holomorphic hermitian vector bundle on the complex torus $\CC^g/(\ZZ^g + \tau \ZZ^g)$ and then to compute its Chern classes. A formal construction of this bundle requires the machinery of vector bundles on abelian varieties and Fourier--Mukai transforms.  In this section we recall some foundational prerequisites on complex tori and vector bundles  on them. A more detailed exposition  can be found in \cite{BirkenhakeLange, Bost, MumfordAb, Polishchuk}.

\subsection{Dual complex torus}  Let $V$ be a finite dimensional complex vector space, $\Gamma \subset V$ a full lattice and $T = V/\Gamma$ the corresponding complex torus.  We put:
$$
V^\ddagger := \Hom_{\bar\CC}(V, \CC) = \left\{V \stackrel{l}\lar \CC \, \big| \, l \; \mbox{is} \, \CC\mbox{-antilinear} \right\}.
$$
We equip $V^\ddagger$ with the structure of a complex vector space by putting 
$(\lambda l)(v) := \lambda \cdot l(v)$ for any $l \in V^\ddagger, v \in V$ and 
$\lambda \in \CC$. It is not difficult to see that the canonical $\RR$-linear pairing 
\begin{equation}\label{E:PairingDagger}
\langle -, \, -\, \rangle: V^\ddagger \times V \lar \RR, \; (l, v) \mapsto \mathfrak{Im}\bigl(l(v)\bigr)
\end{equation}
is non-degenerate. The dual lattice  is defined as
\begin{equation*}
\Gamma^\ddagger : = \left\{l \in V^\ddagger \, \big| \, \langle l, \gamma \rangle \in \ZZ \; \mbox{for all} \; \gamma \in \Gamma \right\}.
\end{equation*}
The dual complex torus is by the definition $\widehat{T} := V^\ddagger/\Gamma^\ddagger$. 

\smallskip
\noindent
Let $T_1 = V_1/\Gamma_1$ and $T_2 = V_2/\Gamma_2$ be two complex tori and 
$T_1 \stackrel{\phi}\lar T_2$ be a holomorphic group homomorphism. Then there exists a unique $\CC$-linear map $V_1 \stackrel{f}\lar V_2$ such that $f(\Gamma_1) \subseteq \Gamma_2$ inducing $\phi$, see e.g. \cite[Proposition 1.2.1]{BirkenhakeLange}. Next, we have an induced $\CC$-linear map
$
V_2^\ddagger \stackrel{f^\ddagger}\lar V_1^\ddagger, \, l \mapsto l \circ f. 
$
It follows that $f^\ddagger(\Gamma_2^\ddagger) \subseteq \Gamma_1^\ddagger$ and, as a consequence, we get the dual morphism of complex tori
 $\widehat{T}_2 \stackrel{\widehat\phi}\lar \widehat{T}_1$.

\smallskip
\noindent
Next, note that we have an isomorphism of vector spaces 
$$
\CC \lar \CC^\ddagger = \Hom_{\bar{\CC}}(\CC, \CC), \; 1 \mapsto \sigma,
$$
where $\sigma(\lambda) = \bar{\lambda}$. Now, let $V = \CC^n$. For any $1 \le j \le n$ we define $\sigma_j \in V^\ddagger$ by setting
$
\sigma_j\left(
\begin{smallmatrix}
\lambda_1 \\
\vdots \\
\lambda_n
\end{smallmatrix}
\right) := \bar{\lambda}_j.
$
Then $(\sigma_1, \dots, \sigma_n)$ is a distinguished basis of $V^\ddagger$ and we get the corresponding  distinguished isomorphism of vector spaces $\CC^n \rightarrow V^\ddagger$. The proof of the following result is straightforward. 
\begin{lemma}\label{L:Dual}
Let $\Upsilon \in \Mat_{m \times n}(\CC)$ and $\CC^n \stackrel{\Upsilon}\lar \CC^m$ be the associated linear map. Then the matrix of the dual map $(\CC^m)^\ddagger \stackrel{\Upsilon^\ddagger}\lar (\CC^n)^\ddagger$ with respect to the above distinguished  bases is $\bar{\Upsilon}^t$. 
\end{lemma}

\smallskip
\noindent
After these preparations we can state and prove the following result.
\begin{proposition}\label{P:DualIsogeny} Let $\tau = s + i t \in \CC$ with $t \ne 0$, $K  \in \Mat_{g \times g}(\ZZ)$ be symmetric and positive definite, $\Omega = \tau K$, $A = \CC^g/(\ZZ^g + \tau \ZZ^g)$, $B = \CC^g/(\ZZ^g + \Omega \ZZ^g)$ and 
\begin{equation}\label{E:MainIsogeny}
A \stackrel{\kappa}\lar B, \; \left[\vec{z}\right] \mapsto \left[K \vec{z}\right].
\end{equation}
Then we have: $\widehat{A} = \CC^g/\frac{1}{t}(\ZZ^g + \tau \ZZ^g)$, 
$\widehat{B} = \CC^g/\frac{1}{t}(K^{-1}\ZZ^g + \tau \ZZ^g)$ and 
\begin{equation}
\widehat{B} \stackrel{\widehat\kappa}\lar \widehat{A}, \; \left[\vec{z}\right] \mapsto  \left[K\vec{z}\right].
\end{equation}
\end{proposition}
\begin{proof}
We compute the dual lattice $\Gamma^\ddagger$ of $\Gamma := \ZZ^g + \Omega \ZZ^g$. First note that $(\sigma_1, i \sigma_1, \dots, \sigma_g, i \sigma_g)$ is a basis of $V^\ddagger$ viewed as an $\RR$-vector space. Let 
$\alpha_1, \dots, \alpha_g, \beta_1, \dots, \beta_g \in \RR$, $$l = \alpha_1 \sigma_1 + \beta_1 i \sigma_1 + \dots + \alpha_g \sigma_g + \beta_g i \sigma_g$$ and 
$\vec{m} = \left(
\begin{smallmatrix}
m_1 \\
\vdots \\
m_g
\end{smallmatrix}
\right) \in \ZZ^g$. Then  we have:
$
\bigl\langle l, \vec{m}\bigr\rangle = \beta_1 m_1 + \dots + \beta_g m_g,
$
where $\langle -, \, -\, \rangle$ is the pairing given by (\ref{E:PairingDagger}). Hence, $l \in \Gamma^\ddagger$ implies that $\beta_1, \dots, \beta_g \in \ZZ$. 

Let $(\vec{e}_1, \dots, \vec{e}_g)$ be the canonical basis of $\ZZ^g$. We have to check the constraint
$$\langle l, (s+it) K\vec{e}_j\rangle \in \ZZ \; \mbox{\rm for all} \; 1 \le j \le g.$$
It is equivalent to the statement that
\begin{equation}\label{E:identity1}
- (\alpha_1 k_{1j} + \dots + \alpha_g k_{gj}) t + (\beta_1 k_{1j} + \dots + \beta_g k_{gj}) \in \ZZ \;\,  \mbox{\rm for all} \;\,  1 \le j \le g,
\end{equation}
where $k_{ij}$ are the entries of the matrix $K$. Consider the $(1 \times g)$ matrix $L := (-t \alpha_1 + s \beta_1, \dots, -t\alpha_g + s\beta_g)$. The constraint (\ref{E:identity1}) is equivalent to $L K \in \ZZ^g$. Since $K = K^t$, we get an equivalent condition
$$
K \cdot (s \vec{\beta} - t \vec{\alpha}) \in \ZZ^g, \; \mbox{\rm where} \; 
\vec{\alpha} = 
\left(
\begin{array}{c}
\alpha_1 \\
\vdots \\
\alpha_g
\end{array}
\right) \, \mbox{\rm and} \;
\vec{\beta} = 
\left(
\begin{array}{c}
\beta_1 \\
\vdots \\
\beta_g
\end{array}
\right).
$$
Let $\vec{u}:= K \cdot (t \vec{\alpha} - s \vec{\beta})$. Then 
$\vec{\alpha} = \dfrac{s \vec\beta + K^{-1} \vec{u}}{t}$. Summing up, we get:
\begin{equation*}
\Gamma^\ddagger = \left\{\left.
\dfrac{s \vec\beta + K^{-1} \vec{u}}{t} + i \beta \, \right| \, \vec{\beta}, \vec{u} \in \ZZ^g
\right\} = \left\{\left.
\dfrac{K^{-1} \vec{u}}{t} + i \dfrac{\tau}{t}\beta \, \right| \, \vec{\beta}, \vec{u} \in \ZZ^g
\right\} = \dfrac{1}{t}\left(K^{-1} \ZZ^g + \tau \ZZ^g\right). 
\end{equation*}
This provides  an explicit  description of $\widehat{B}$ as well as of $\widehat{A}$ (we may just put $\Omega = \tau I_g)$. The remaining part of the proof follows from Lemma \ref{L:Dual}.
\end{proof}

\subsection{Line bundles on complex tori and Fourier--Mukai transforms} Let $U$ be a finite dimensional complex vector space. Recall that an $\RR$-bilinear form
$U \times U \stackrel{H}\lar \CC$ is hermitian, if it is $\CC$-linear in the first argument and satisfies $H(u_1, u_2) = \overline{H(u_2, u_1)}$ for all $u_1, u_2 \in U$ (no non-degeneracy condition is imposed). For such a form $H$  we denote by $E = \mathfrak{Im}(H): U \times U \rightarrow \RR$ its imaginary part.  Note that the bilinear form $E$ is skew-symmetric. 

In what follows, we denote $S = \left\{z \in \CC \, \big| \, |z| = 1 \right\}$. 
As in the previous subsection, let $V$ be a finite dimensional complex vector space, $\Gamma \subset V$ a full lattice and $T = V/\Gamma$ the corresponding complex torus. Recall (see e.g. \cite[Section 2.2]{BirkenhakeLange}) that a pair $(H, \alpha)$ is called an \emph{Appell--Humbert datum} if the following conditions are fulfilled: 
\begin{itemize}
\item $V \times V \stackrel{H}\lar \CC$ is a hermitian form such that $E(\Gamma \times \Gamma) \subseteq \ZZ$. 
\item $\Gamma \stackrel{\alpha}\lar S$ satisfies $\alpha(\gamma_1 + \gamma_2) = (-1)^{E(\gamma_1, \gamma_2)} \alpha(\gamma_1) \alpha(\gamma_2)$ for all $\gamma_1, \gamma_2 \in \Gamma$. 
\end{itemize}
The set of all Appell--Humbert data forms an abelian group with respect to the operation $(H_1, \alpha_2) + (H_2, \alpha_2) = (H_1 + H_2, \alpha_1 \alpha_2)$, which is isomorphic to the Picard group $\Pic^0(T)$ of $T$, see \cite[Theorem 2.2.3]{BirkenhakeLange}. We denote by $\kL(H, \alpha)$ the line bundle attached to an Appell--Humbert datum $(H, \alpha)$. 

We have a group homomorphism $V^\ddagger \lar \Hom(\Gamma, S), l \mapsto \exp\bigl(2\pi i \langle l, \, -\,\rangle \bigr)$, whose kernel is the lattice $\Gamma^\ddagger$. By \cite[Proposition 2.4.1]{BirkenhakeLange} we have induced isomorphisms
\begin{equation}\label{E:IdentificationDualTorus}
\widehat{T} = V^\ddagger/\Gamma^\ddagger \stackrel{\cong}\lar \Hom(\Gamma, S) \stackrel{\cong}\lar \Pic^0(T),
\end{equation}
where the second  map assigns to $\alpha \in \Hom(\Gamma, S)$ the line bundle $\kL(0, \alpha)$. In these terms, points of $\widehat{T}$ are got identified with topologically trivial line bundles on $T$. 

Consider now the following \emph{universal} Appell--Humbert datum $(H_c, \alpha_c)$ 
$$
(V^\ddagger \times V) \times (V^\ddagger \times V) \stackrel{H_c}\lar \CC, \; 
\bigl((l_1, v_1), (l_2, v_2)\bigr) \mapsto l_1(v_2) + \overline{l_2(v_1)}
$$
and 
$
\Gamma^\ddagger \times \Gamma \stackrel{\alpha_c}\lar \ZZ, (l, v) \mapsto \exp\bigl(\pi i l(v)\bigr),
$
which defines the so-called \emph{Poincar\'e line bundle} $\kP = \kL(H_c, \alpha_c)$ on the complex torus $\widehat{T} \times T = (V^\ddagger \times V)/(\Gamma^\ddagger \times \Gamma)$. This line bundle is uniquely characterized by the following two properties (see e.g.~\cite[Theorem 2.5.1]{BirkenhakeLange})
\begin{itemize}
\item $\kP\Big|_{{[\xi]} \times T} \cong \kL\bigl(0, \exp\bigl(2\pi i \langle\xi, \, -\,\rangle \bigr)\bigr)$ for any $\xi \in V^\ddagger$ (see also the  isomorphism (\ref{E:IdentificationDualTorus})).
\item $\kP\Big|_{\widehat{T} \times \{0\}} \cong \kO_{\widehat{T}}$.
\end{itemize}
In other words, $\kP$ is a normalized universal family of topologically trivial line bundles on the complex torus $T$.

From now on we assume that our complex tori $T$ are  abelian varieties, see \cite[Chapter IV]{BirkenhakeLange} for their definitions and characterizations. Let $T \stackrel{\pi}\longleftarrow \widehat{T} \times T \stackrel{\widehat\pi}\lar \widehat{T}$ be the canonical projections. The following classical  result is due to Mukai \cite[Theorem 2.2]{Mukai}. 
\begin{theorem}
The functor
\begin{equation}
D^b\bigl(\Coh(T)\bigr) \stackrel{\FF_T}\lar D^b\bigl(\Coh(\widehat{T})\bigr), \; 
\kF^\bu \mapsto R\widehat{\pi}_\ast \bigl(\kP \otimes \pi^\ast(\kF^\bu)\bigr)
\end{equation}
is an equivalence of triangulated categories. 
\end{theorem}
The functor $\FF_T$ (called \emph{Fourier--Mukai transform}) has a number of remarkable properties, see \cite{Mukai} as well as \cite{BirkenhakeLange, Polishchuk}. Of special interest for us is the case $A = \CC^g/(\ZZ^g + \tau \ZZ^g)$ and $B = \CC^g/(\ZZ^g + \Omega \ZZ^g)$. Recall the following result of Mukai \cite{Mukai}.
\begin{proposition}\label{P:FMTdualIsogeny}
Consider the isogeny $A \stackrel{\kappa}\lar B, [\vec{z}] \mapsto [K \vec{z}]$. Then the following diagram of triangulated categories and functors is commutative:
\begin{equation}
\begin{array}{c}
\xymatrix{
D^b\bigl(\Coh(B)\bigr) \ar[rr]^-{\FF_B} \ar[d]_-{\kappa^\ast} &&  D^b\bigl(\Coh(\widehat{B})\bigr) \ar[d]^{\widehat{\kappa}_\ast}\\
D^b\bigl(\Coh(A)\bigr)  \ar[rr]^-{\FF_A} &&  D^b\bigl(\Coh(\widehat{A})\bigr), \\
}
\end{array}
\end{equation}
where $\widehat{B} \stackrel{\widehat\kappa}\lar \widehat{A}$ is the dual isogeny (explicitly described in Proposition \ref{P:DualIsogeny}). 
\end{proposition}

\smallskip
\noindent
Let $\kT = \kT_{\vec{0}}$ be the line bundle on $B$ introduced in Section \ref{S:WaveFunctions}. It is characterized by the property that the theta function $\Theta[\vec{0}, \vec{0}]\bigl(\vec{z} \, \big| \, \Omega\bigr)$ is its unique section up to the   multiplication with a scalar. The first Chern class of $\kT$ is given by the differential 2-form
\begin{equation}\label{E:ChernClassTheta}
\omega = \frac{i}{2t} \sum\limits_{p, q = 1}^g \lambda_{pq} dz_{p} \wedge d\bar{z}_q,
\end{equation}
where $(\lambda_{pq})_{p, q = 1}^g = K^{-1}$ and $(z_1, \dots, z_g)$ are the standard coordinates on $B$, see for instance \cite[page 196]{Bost} for a short proof. Using \cite[Theorem 1.3]{Polishchuk} (see also  \cite[Exercise 2.6(2)]{BirkenhakeLange}) we conclude that the matrix of the hermitian form
$\CC^g \times \CC^g \stackrel{H}\lar \CC$ from the Appell--Humbert datum defining the line bundle $\kT$ is  $\Omega^{-1} = \frac{1}{t} K^{-1}$. Next, $\kT$ is a principal polarization on $B$, see \cite[Section 4.1]{BirkenhakeLange}. As a consequence, we have an isomorphism of abelian varieties 
$$
\phi = \phi_\kT: B \lar \widehat{B}, \; b \mapsto t^*_b(\kT) \otimes \kT^{-1},
$$
where $B \stackrel{t_b}\lar B, x \mapsto x + b$; see \cite[Section 4.1]{BirkenhakeLange}.

\begin{proposition}\label{P:FMTproperties}
 The following results are true. 
 \begin{enumerate}
 \item In the notation of Proposition \ref{P:DualIsogeny}, the isomorphism  $B \stackrel{\phi}\lar \widehat{B}$ is given by the map $\CC^g \lar \left(\CC^g\right)^\ddagger, \vec{v} \mapsto H(\vec{v}, \,-\,)$, where the matrix of the hermitian form $H$ in the standard basis is $\Omega^{-1}$.
 \item Consider the auto-equivalence $\widetilde{\FF}_B$ given by the composition
 $$
 D^b\bigl(\Coh(B)\bigr) \stackrel{\FF_B}\lar D^b\bigl(\Coh(\widehat{B})\bigr) \stackrel{\phi^\ast}\lar D^b\bigl(\Coh(B)\bigr). 
 $$
 Then we have: $\widetilde{\FF}_B(\kT) \cong \kT^\vee$. 
 \end{enumerate}
\end{proposition}

\begin{proof}
For the  first statement, see \cite[Lemma 2.4.5]{BirkenhakeLange}, whereas for the second statement, see \cite[Theorem 3.13]{Mukai}. 
\end{proof}

\smallskip
\noindent
In what follows, we use the identifications
\begin{equation}\label{E:AandDualA}
\widehat{A} \stackrel{\psi}\lar A,  \; [\vec{z}] \mapsto t  [\vec{z}] \quad 
\end{equation}
and $
B \stackrel{\phi}\lar \widehat{B}$,
where $\widehat{A}$ and $\widehat{B}$ were explicitly described in Proposition \ref{P:DualIsogeny}.
\begin{theorem}\label{T:FMTCenterMass} Let $D^b\bigl(\Coh(A)\bigr) \stackrel{\FF}\lar D^b\bigl(\Coh(A)\bigr)$ be the composition $\psi_\ast \circ \FF_A$, $A \stackrel{\kappa}\lar B$ be the isogeny (\ref{E:MainIsogeny}) and $\kR := \kappa^\ast(\kT)$. Then the following results are true. 
\begin{enumerate}
\item In terms of the above identifications of abelian varieties and their duals, the ``dual isogeny'' $$\widetilde{\kappa}: \; B = \CC^g/(\ZZ^g + \Omega \ZZ^g) \stackrel{\phi}\lar \widehat{B}
\stackrel{\widehat{\kappa}}\lar  \widehat{A} \stackrel{\psi}\lar A =  \CC^g/(\ZZ^g + \tau \ZZ^g)$$  is given by $\widetilde{\kappa}([\vec{z}]) =  [\vec{z}]$. 
\item $\kC := \FF(\kR)$ is a simple semi-homogeneous vector bundle of rank $\delta = \det(K)$. 
\item We have: 
\begin{equation}\label{E:ChernClassMagnetic}
c_1(\kC) = -\frac{i}{2t} \sum\limits_{p, q = 1}^g K^\sharp_{pq} dw_{p} \wedge d\bar{w}_q,
\end{equation}
where $K^\sharp$ is the adjunct matrix of $K$ (i.e.  $K^\sharp K = \delta I_g$) and $(w_1, \dots, w_g)$ are the standard (local) coordinates on $A$. 
\end{enumerate}
\end{theorem}
\begin{proof}  (1) The description of $\widetilde\kappa$ follows from the explicit descriptions of $\widehat{\kappa}$ and $\phi$ given in Proposition \ref{P:DualIsogeny} and Proposition \ref{P:FMTproperties}.

\smallskip
\noindent
(2) Using Proposition \ref{P:FMTdualIsogeny} and Proposition \ref{P:FMTproperties}, we get the following isomorphisms:
$$
\kC = \FF(\kR) = \FF\bigl(\kappa^\ast \kT\bigr) \cong \widetilde{\kappa}_\ast \widetilde{\FF}_B(\kT) \cong \widetilde{\kappa}_\ast(\kT^\vee).
$$
Since $\widetilde\kappa$ is a finite \'etale morphism of degree $\delta$, $\kC$ is a vector bundle of rank $\delta$. Since $\End_A(\kR) \cong \CC$ and $\FF$ is an auto-equivalence, we conclude that $\End_A(\kC) = \CC$, too. The fact that $\kC$ is semi-homogeneous (meaning that for any $a \in A$ there exists a line bundle $\kL_a \in \Pic(A)$ such that $t_a^\ast(\kC) \cong \kC \otimes \kL_a$) follows from \cite[Theorem 5.8]{MukaiSemiHom}.

\smallskip
\noindent
(3) Let $L  = \bigl\{b_1, \dots, b_\delta\bigr\}$ be the kernel of the isogeny
$B \stackrel{\widetilde\kappa}\lar A$ and for any $1 \le i \le \delta$ let
$B \stackrel{t_i}\lar B, b \mapsto b + b_i$ be the corresponding translation map. It is easy to see that  the following commutative diagram 
\begin{equation}\label{P:PullBack}
\begin{array}{c}
\xymatrix{
\protect\underbrace{B \sqcup \dots \sqcup B}_{\delta \; \mbox{\scriptsize{\sl times}}} \ar[rr]^-{(\mathsf{id}, \dots, \mathsf{id})} \ar[d]_-{t = (t_1, \dots, t_\delta)} & & B \ar[d]^-{\widetilde\kappa} \\
B \ar[rr]^-{\widetilde\kappa} & & A
}
\end{array}
\end{equation}
is a pull-back diagram. Since all morphisms in (\ref{P:PullBack}) are finite, we have base-change isomorphism
$
\widetilde{\kappa}^\ast\bigl(\kC\bigr) = \widetilde{\kappa}^\ast \widetilde{\kappa}_\ast  \kT^\vee \cong \bigoplus\limits_{i = 1}^\delta 
t_i^\ast\bigl(\kT^\vee\bigr).
$
Using the formula (\ref{E:ChernClassTheta}) for the first Chern class of $\kT$ we get:
$
\widetilde{\kappa}^\ast\bigl(c_1(\kC)\bigr) = c_1\bigl(\widetilde{\kappa}^\ast\bigl(\kC\bigr)\bigr) = -\delta c_1(\kT) = -\frac{i}{2t} \sum\limits_{p, q = 1}^g K^\sharp_{pq} dz_{p} \wedge d\bar{z}_q.
$
It remains to recall that by part (1), the isogeny $\widetilde\kappa$ is given by the identity map $\CC^g \rightarrow \CC^g$. Hence, $\widetilde\kappa$ identifies (local) standard coordinates $(w_1, \dots, w_g)$ and $(z_1, \dots, z_g)$ on the complex tori $A$ and $B$, implying the formula (\ref{E:ChernClassMagnetic}).
\end{proof}

\begin{remark} Since $\kC$ is semi-homogeneous, its total Chern class can be expressed via its first Chern class:
\begin{equation}
c(\kC) = \left(1 + \frac{c_1(\kC)}{\delta}\right)^\delta,
\end{equation}
see for instance \cite[Theorem 5.12]{Yang}. Moreover, $\kC$ can be equipped with a hermitian metric and a projectively flat hermitian connection; see Remark \ref{R:ProjFlat}. We shall see later on that these structures  can be made more concrete. 
\end{remark}

\subsection{Poincar\'e line bundle revisited} 
 Consider the meromorphic function $\CC^2 \stackrel{Z}\lar \CC$ given by the formula
$Z(\xi, z) = \dfrac{\vartheta(\xi) \vartheta(z)}{\vartheta(z-\xi)}$ (up to a normalization constant, it is the inverse of the so-called Kronecker elliptic function). It follows from (\ref{E:ThetaFunctionsTransfRules}) that
$$
Z(\xi+1, z) = Z(\xi, z) = Z(\xi, z+1), \quad Z(\xi+\tau, z) = \exp(-2\pi i z) Z(\xi, z) .
$$
and $
Z(\xi, z+\tau) = \exp(-2\pi i \xi) Z(\xi, z)$. As a consequence, $Z$ is a meromorphic section of the line bundle $\kQ = \bigl((\CC \times \CC) \times\CC \bigr)/\sim$ on $E \times E$, given by the automorphy factor specified by the rules
\begin{equation}
\left\{
\begin{array}{l}
\bigl((\xi +1, z), v\bigr) \sim  \bigl((\xi, z), v\bigr) \sim \bigl((\xi, z+1), v\bigr) \\
\bigl((\xi +\tau, z), v\bigr) \sim  \bigl((\xi, z), \exp(-2\pi i z)v\bigr) \\
\bigl((\xi, z +\tau), v\bigr) \sim  \bigl((\xi, z), \exp(-2\pi i \xi)v\bigr). \\
\end{array}
\right.
\end{equation}
Note that $\kQ \cong \kO_{E \times E}\bigl(\{0\}\times E + E \times \{0\} - \Xi\bigr)$, where $\Xi \subset E \times E$ is the diagonal. Next,
\begin{itemize}
\item for any $\xi = a\tau + b  = E$ we have: 
$\kQ\Big|_{\{\xi\} \times E} \cong \kO_E\bigl([0]-[\xi]\bigr) \cong \kL_{0, \xi}$, see also (\ref{E:LBTorus}).
\item $\kQ\Big|_{E \times \{0\}} \cong \kO_{E}$.
\end{itemize}
In this way, we  identify the complex torus $E = \CC/\langle 1, \tau\rangle$, the corresponding dual torus $\widehat{E}$ and the Jacobian variety $\Pic^0(E)$ and regard $\kQ$ as the normalized Poincar\'e line bundle on $\widehat{E} \times E = E \times E$. 

\smallskip
\noindent
Let $z = x + \tau y$ and $\xi = \beta + \alpha \tau$, where $x, y, \alpha, \beta\in \RR$. Similarly to (\ref{E:MetricLB1}),  consider the function $\CC^2 \stackrel{h}\lar \RR_{> 0}$ given by the formula
\begin{equation}\label{E:MetricLB2}
h(\xi, z)  := \exp(- 4 \pi a ty).
\end{equation}
Analogously  to Lemma \ref{L:MetricOnLB}, we have the following result. 
\begin{lemma} The above function $h$ defines a hermitian metric on the line bundle $\kQ$. Moreover, for  any $\xi \in \widehat{E}$, the induced metric on $\kQ\big|_{\{\xi\} \times E} \cong \kL_{0, \xi}$ coincides with the one introduced in Lemma \ref{L:MetricOnLB}.
\end{lemma}
\begin{proof} It is a straightforward computation that for any smooth global section $f$ of the line bundle $\kQ$,  we get  a smooth function
\begin{equation*}
\widehat{E} \times E \lar \RR_{>0}, \; (\xi, z)  \mapsto  \big\lVert f(\xi, z) \bigr\rVert^2 := \big|f(\xi, z)\big|^2 \,h(\xi, z).
\end{equation*}
It shows that $h$ indeed defines a hermitian metric on $\kQ$. The second result follows from the defining formulae (\ref{E:MetricLB2}) and (\ref{E:MetricLB1}).
\end{proof}

\smallskip
\noindent
Finally, we shall need the following well-known result.
\begin{lemma}\label{L:PoincareSumAndProduct} Let $T := 
\underbrace{E \times \dots \times E}_{n \; \mbox{\scriptsize{\sl times}}}$. For $1 \le j \le n$, consider the canonical projection map
$
\widehat{E} \times T \stackrel{p_j}\lar  \widehat{E} \times E, \; (\xi, z_1, \dots, z_n) \mapsto (\xi, z_j).
$
Then we have an isomorphism: 
\begin{equation}
p_1^\ast \kQ \otimes \dots \otimes p_n^\ast\kQ \cong (\mathsf{id} \times \mu_n)^\ast \kQ,
\end{equation}
where $T \stackrel{\mu_n}\lar E, (z_1, \dots, z_n) \mapsto z_1 +  \dots + z_n$. 
\end{lemma}

\begin{proof} This isomorphism can be for instance deduced from the theorem of the cube \cite[Section II.6]{MumfordAb}.
\end{proof}

\begin{lemma} For any $1 \le i \le g$ consider the morphism
$$
A \times A = \underbrace{E \times \dots \times E}_{g \; \mbox{\scriptsize{\sl times}}} \times \underbrace{E \times \dots \times E}_{g \; \mbox{\scriptsize{\sl times}}} \xrightarrow{\pi_{i, g+i}} E \times E, \; (\xi_1, \dots, \xi_g; z_1, \dots, z_g) \mapsto (\xi_i, z_i). 
$$
Then we have: 
$
(\psi \times \mathsf{id})_\ast(\kP)  \cong \bigotimes\limits_{i = 1}^g \pi_{i, g+i}^\ast(\kQ),
$
where the morphism $\psi$ was defined by (\ref{E:AandDualA}). 
\end{lemma}

\begin{proof} This fact is a consequence of 
\cite[Exercise II.6.11]{BirkenhakeLange}.
\end{proof}

\section{Magnetic vector bundle of the multilayer torus model of FQHE}\label{S:Multilayer}

\noindent
In what follows, we shall use the identifications $\widehat{E} \stackrel{\psi}\lar E$ and $\widehat{A} \stackrel{\psi}\lar A$, where $\psi$ is given by the formula  (\ref{E:AandDualA}). For any $1 \le k \le g$ and $1 \le p \le n_k$ consider the projection morphism
\begin{equation*}
\widehat{A} \times X \xrightarrow{\pi_p^{(k)}} \widehat{E} \times E, \quad (\vec\xi, \vec{z}) \mapsto \bigl(\xi_k, z_p^{(k)}\bigr).
\end{equation*}
Let $\widehat{E} \times E \stackrel{\pi_E}\lar E$ be the canonical projection. 
Consider the following line bundle $\kE$ on the product $\widehat{A}\times X$:
\begin{equation}
\kE := \left\{
\begin{array}{cl}
\bigotimes\limits_{k = 1}^g \bigotimes\limits_{p = 1}^{n_k} \bigl(\pi_p^{(k)} \bigr)^\ast \Bigl(\kQ \otimes \pi_E^\ast\bigl(\kL_{d, 0}\bigr)\Bigr)
&  \, \mbox{\textrm if} \; \;   \epsilon(K) + d \in 2 \NN \\
\bigotimes\limits_{k = 1}^g \bigotimes\limits_{p = 1}^{n_k} \bigl(\pi_p^{(k)} \bigr)^\ast \Bigl(\kQ \otimes \pi_E^\ast\bigl(\kL^\sharp_{d, 0}\bigr)\Bigr) &  \, \mbox{\textrm if} \; \; \epsilon(K) + d \in 2 \NN +1.
\end{array}
\right. 
\end{equation}
It is clear that for any $\vec{\xi} \in \widehat{A}$ we have: $\kE\Big|_{\{\vec\xi\}\times X} \cong \kW_{K, \vec{\nn}, \vec{\xi}}$ (the latter line bundle was defined by  (\ref{E:linebundleNbody})). It follows from the K\"unneth formula that
\begin{equation}\label{E:Cohomology}
H^i\bigl(X, \kW_{K, \vec{\nn}, \vec{\xi}}\bigr) = \left\{
\begin{array}{cl} 
\mathsf{W}_{K, \vec{\nn}, \vec{\xi}} & \; \mbox{if} \, \; i = 0 \\
0 & \; \mbox{if} \, \; i \ge 1; 
\end{array}
\right.
\end{equation}
see also (\ref{E:FullSpaceOfWaveFunctions}) for the description of the space $\mathsf{W}_{K, \vec{\nn}, \vec{\xi}}$. 
\begin{lemma}\label{L:SpectralBundle} Let $\widehat{A} \stackrel{\widehat{\pi}}\longleftarrow \widehat{A} \times X \stackrel{{\pi}}\lar X$ be the canonical projections and 
$\kA := \widehat{\pi}_\ast(\kE)$. Then the following results are true.
\begin{enumerate}
\item For any $\vec{\xi} \in \widehat{A}$, the canonical base-change morphism
$$
\kA\Big|_{\vec{\xi}} \lar \Gamma\bigl(X, \kW_{K, \vec{\nn}, \vec{\xi}}\bigr) = \mathsf{W}_{K, \vec{\nn}, \vec{\xi}}
$$
is an isomorphism. In particular $\kA$ is a hermitian vector bundle on $\widehat{A}$. 
\item We have isomorphisms
\begin{equation*}
\kE \cong  \left\{
\begin{array}{cl}
(\mathsf{id} \times \mu)^\ast \kP \otimes \pi^\ast(\kL_{d, 0} \boxtimes \dots \boxtimes \kL_{d, 0})
&  \, \mbox{\textsl if} \; \;   \epsilon(K) + d \in 2 \NN \\
(\mathsf{id} \times \mu)^\ast \kP \otimes \pi^\ast(\kL^\sharp_{d, 0} \boxtimes \dots \boxtimes \kL^\sharp_{d, 0}) &  \, \mbox{\textsl if} \; \; \epsilon(K) + d \in 2 \NN +1,
\end{array}
\right. 
\end{equation*}
where the morphism $X \stackrel{\mu}\lar A$ was defined in Lemma \ref{L:sectionsofpullback2}.
\end{enumerate}
\end{lemma}
\begin{proof} (1) It follows from (\ref{E:Cohomology}) that the assumptions of \cite[Corollary II.5.3]{MumfordAb} are satisfied, implying that the base-change map is indeed an isomorphism. Hence, $\kA$ is a vector bundle on $\widehat{A}$.  Next, the vector  space $\mathsf{W}_{K, \vec{\nn}, \vec{\xi}}$ is equipped with the hermitian vector product induced by the scalar product of one-particle system (\ref{E:scalarproduct}), which yields a hermitian metric on $\kA$.

\smallskip
\noindent
(2) This isomorphism follows from the definition of $\kE$ and Lemma \ref{L:PoincareSumAndProduct}.
\end{proof}

\begin{proposition} Let $\Xi_K$ be the divisor on $X$ defined by (\ref{E:WenDivisor}) and $\overline{\Xi}_K = \widehat{A} \times \Xi_K$ be the corresponding divisor on $\widehat{A} \times X$. We put 
$\kF:= \kE\bigl(-\bigl[\overline{\Xi}_K\bigr]\bigr)$. Then we have:
\begin{equation}\label{E:DerivedDirectImage}
R\widehat{\pi}_\ast(\kF) \cong 
\bigoplus\limits_{j = 0}^{n-g}\bigl(\kC[-j]\bigr)^{\oplus \binom{n-g}{j}},
\end{equation} 
where $\kC$ is the vector bundle introduced in Theorem \ref{T:FMTCenterMass}.
\end{proposition}

\begin{proof} It follows from Lemma \ref{L:SpectralBundle} (part (2)) and  Lemma \ref{L:KWwavefunction} that
\begin{equation}\label{E:sheafFdescription}
\kF \cong (\mathsf{id} \times \mu)^\ast\bigl(\kP \otimes \pi_A^\ast(\kR)\bigr).
\end{equation}
Now we prove the statement by induction on $\vec{\nn} = (n_1, \dots, n_g)$. The case
$\vec{\nn} = (1, \dots, 1)$ is trivial. Without loss of generality we may assume that $n_1 \ge 2$. Consider the morphism $X \stackrel{\varrho}\lar X$ given by the rule
$$
\bigl(z_1^{(1)}, z_2^{(1)}, \dots, z_{n_1}^{(1)}; \dots, z_1^{(g)}, \dots, z_{n_g}^{(g)}\bigr) \stackrel{\varrho}\mapsto \bigl(z_1^{(1)}, z_2^{(1)} + z_1^{(1)}, \dots, z_{n_1}^{(1)}; \dots, z_1^{(g)}, \dots, z_{n_g}^{(g)}\bigr).
$$
Obviously, the diagram
$$
\xymatrix{
\widehat{A} \times X \ar[rr]^-{\mathsf{id} \times \varrho} \ar[rd] _-{\widehat{\pi}}& &  \widehat{A} \times X \ar[ld]^-{\widehat{\pi}} \\
& \widehat{A} & 
}
$$
is commutative. As a consequence, we have an isomorphism of functors
\begin{equation}\label{E:IsomBanal}
R\widehat{\pi}_\ast \cong R\widehat{\pi}_\ast \cdot (\mathsf{id} \times \varrho)^\ast.
\end{equation}
Let 
$
\breve{X} = \underbrace{E \times \dots \times E}_{n_1-1  \, \mbox{\scriptsize{\sl times}}} \times \dots \times \underbrace{E \times \dots \times E}_{n_g  \, \mbox{\scriptsize{\sl times}}} \stackrel{\breve\mu}\lar A,
$
where $\breve\mu = 
\left(\mu_{n_1-1} \times \mu_{n_2} \times \dots \times \mu_{n_g}\right)$ and 
$X \stackrel{\bar\pi}\lar \overline{X}, \bigl(z_1^{(1)}, z_2^{(1)}, \dots, z_{n_1}^{(1)}; \dots, z_1^{(g)}, \dots, z_{n_g}^{(g)}\bigr) \mapsto \bigl(z_2^{(1)}, \dots, z_{n_1}^{(1)}; \dots, z_1^{(g)}, \dots, z_{n_g}^{(g)}\bigr)$. It is obvious that the following diagram
$$
\xymatrix{
X \ar[r]^-\mu \ar[d]_-\varrho & A  \\
X \ar[r]^-{\bar\pi}  & \overline{X} \ar[u]_-{\breve\mu} 
}
$$
is commutative. As a consequence, we have an isomorphism of functors
$$
(\mathsf{id} \times \mu)^\ast \cong (\mathsf{id} \times \varrho)^\ast (\mathsf{id} \times \bar\pi)^\ast (\mathsf{id} \times \breve\mu)^\ast,
$$
where $\widehat{A} \stackrel{\mathsf{id}}\lar \widehat{A}$ is the identity map. 
From (\ref{E:IsomBanal}) we get an isomorphism of functors
$$
R\widehat{\pi}_\ast (\mathsf{id} \times \mu)^\ast \cong R\widehat{\pi}_\ast 
\widetilde\pi^\ast (\mathsf{id} \times \breve\mu)^\ast,
$$
where $\widetilde\pi = \mathsf{id} \times \bar\pi$. Let $\widehat{A} \times \overline{X} \stackrel{\breve\pi}\lar \widehat{A}$ be the canonical projection. 
Since the diagram 
$$
\xymatrix{
\widehat{A} \times X \ar[rr]^-{\widetilde\pi} \ar[rd] _-{\widehat{\pi}}& &  \widehat{A} \times \overline{X} \ar[ld]^-{\breve{\pi}} \\
& \widehat{A} & 
}
$$
is commutative, we have: $R\widehat{\pi}_\ast \cong R\breve{\pi}_\ast \, 
R\widetilde{\pi}_\ast$. We have (a non-canonical) isomorphism $R\Gamma(E, \kO_E) \cong \CC \oplus \CC[-1]$, which induces an isomorphism of functors
$$
R\widehat{\pi}_\ast (\mathsf{id} \times \mu)^\ast \cong R\breve{\pi}_\ast 
R\widetilde{\pi}_\ast
\widetilde\pi^\ast (\mathsf{id} \times \breve\mu)^\ast \cong 
R\breve{\pi}_\ast (\mathsf{id} \times \breve\mu)^\ast  \oplus R\breve{\pi}_\ast (\mathsf{id} \times \breve\mu)^\ast[-1]. 
$$
Using the description of $\kF$ given by (\ref{E:sheafFdescription}), it implies the statement. 
\end{proof} 

\begin{theorem}\label{T:MagneticBundle}
Let $\kB := \widehat{\pi}_\ast(\kF)$. Then the following results are true.
\begin{enumerate}
\item $\kB$ is a locally free sheaf of rank $\delta$ on $\widehat{A}$ (called \emph{magnetic vector bundle of the multi-layer torus model of FQHE}) isomorphic to the Fourier--Mukai transform $\kC$ of the line bundle $\kR$ on $A$. In particular, $\kB$ is simple and semi-homogeneous. Moreover, its first Chern class is given by the formula
\begin{equation}\label{E:FirstChernClass}
c_1(\kB) = -\frac{i}{2t} \sum\limits_{p, q = 1}^g K^\sharp_{pq} d\xi_{p} \wedge d\bar{\xi}_q,
\end{equation}
where $(\xi_1, \dots, \xi_g)$ are the standard (local) holomorphic coordinates on $\widehat{A}$. 
Next, $\kB$ is a subbundle of the hermitian vector bundle $\kA$ introduced in Lemma \ref{L:SpectralBundle} and it  carries the restricted hermitian metric of $\kA$.
\item For any $\vec{\xi} \in \widehat{A}$, the canonical base-change morphism
$
\kB\Big|_{\vec\xi} \longrightarrow  \Gamma\Bigl(X, \kF\big|_{\{\vec\xi\}\times X}\Bigr)
$
is an isomorphism. 
As a consequence,  we have a natural isomorphism
\begin{equation}
\kB\Big|_{\vec\xi} \lar \mathsf{V}_{K, \vec{\nn}, \vec\xi} \subset \mathsf{W}_{K, \vec{\nn}, \vec{\xi}},
\end{equation}
where $\mathsf{V}_{K, \vec{\nn}, \vec\xi}$ is the space of wave functions of Keski-Vekkuri and Wen (\ref{E:SpaceKV-WenWaveFunctions}) and $\mathsf{W}_{K, \vec{\nn}, \vec{\xi}}$ is the ground state  (\ref{E:FullSpaceOfWaveFunctions}) of our many-body system. In particular, the wave functions $(\Phi_{\vec{c}})_{\vec{c} \in \Pi}$ given by (\ref{E:KV-Wen-wavefunct}) 
provide a holomorphic frame of $\kB$ in a neighbourhood of $\vec{0}$. 
\end{enumerate}
\end{theorem}
\begin{proof} (1) It follows from (\ref{E:DerivedDirectImage}) that $\widehat{\pi}_\ast(\kF) \cong \kC$. Hence, simplicity and semi-homogeneity of $\kB$ as well as the formula for its first Chern class follow  from Theorem \ref{T:FMTCenterMass}.

Consider the canonical inclusion 
$\kE\bigl(- \bigl[\overline{\Xi}_K\bigr]\bigr) \stackrel{\imath}\lar \kE.
$
Since $\widehat{\pi}_\ast$ is left exact,  we get an induced inclusion  $
\kB  \xrightarrow{\widehat{\pi}_\ast(\imath)} \kA,$ where 
$\kA = \widehat{\pi}_\ast(\kE)$ and $\kB = \widehat{\pi}_\ast\bigl(\kE\bigl(- \bigl[\overline{\Xi}_K\bigr]\bigr) \cong \widehat{\pi}_\ast(\kF)$.
By Lemma \ref{L:SpectralBundle}, the vector bundle $\kA$ carries a natural hermitian metric, which provides a hermitian metric for its subbundle $\kB$.

\smallskip
\noindent
(2) According to  (\ref{E:DerivedDirectImage}), the coherent sheaf  $R^i\widehat{\pi}_\ast(\kF)$ is locally free for all $i \in \NN_0$. By \cite[Corollary II.5.2]{MumfordAb}, the base change map $
\kB\Big|_{\vec\xi} \lar \Gamma\Bigl(X, \kF\big|_{\{\vec\xi\}\times X}\Bigr)
$
is an isomorphism.  It follows that for any  $\vec\xi  \in \widehat{A}$, we have a commutative diagram
\begin{equation}\label{E:MagneticBundleFibers}
\begin{array}{c}
\xymatrix{
\kB\Big|_{\vec{\xi}} \ar@{_{(}->}[d]_-{\imath_{\vec\xi}} \ar[r]^-{\cong}  & 
 \Gamma\Bigl(X, \kW_{K, \vec{\nn}, \vec{\xi}}\bigl(-[\Xi_K]\bigr)\Bigr)
\ar@{_{(}->}[d] \ar[r]^-\cong & \mathsf{V}_{K, \vec{\nn}, \vec{\xi}}  \ar@{^{(}->}[d] \\
\kA\Big|_{\vec\xi}  \ar[r]^-\cong & 
\Gamma\bigl(X, \kW_{K, \vec{\nn}, \vec{\xi}}\bigr) \ar[r]^-\cong & \mathsf{W}_{K, \vec{\nn}, \vec{\xi}}, \\
}
\end{array} 
\end{equation} 
see also Remark \ref{R:CenterMassAllWave}.
\end{proof}

\begin{remark}
The holomorphic vector bundle $\kB$ is determined (up to non-canonical isomorphisms!) by the matrix $K$ alone and does not depend on the choice of the vector $\vec{\nn}$ specifying the number of particles of our many-body system. On the other hand, $\kB$ carries a natural hermitian metric (coming from the scalar product in the Hilbert space of the underlying many-body problem), which depends on the choice of $\vec{\nn}$. 
\end{remark}

\begin{proposition}\label{P:ProjFlat} Let $\bigl(E, (K, \vec{\nn})\bigl)$ be a KV-W datum with a primary matrix $K$. Then the canonical Bott--Chern  connection of the holomorphic hermitian vector bundle $(\kB, h)$ is projectively flat. 
\end{proposition}

\begin{proof} The holomorphic hermitian vector bundle $(\kB, h)$ on the dual torus $\widehat{A}$ was defined 
in Theorem \ref{T:MagneticBundle}. By a result of Bott and Chern \cite[Proposition 3.2]{BottChern} there exists a unique hermitian connection 
$\kB \stackrel{\nabla}\lar  \kB \otimes \Omega^1_{\widehat{A}}$ such that
$\nabla^{(0, 1)} = \bar\partial_{\kB}$, where $\kB \stackrel{\bar\partial_{\kB}}\lar  \kB \otimes \Omega^{0, 1}_{\widehat{A}}$ is the differential operator, which determines   the holomorphic structure on the complex vector bundle $\kB$. 

By Theorem \ref{T:MagneticBundle},  
$F = \bigl(\Phi_1, \dots, \Phi_\delta\bigr)$ is a \emph{holomorphic} frame of $\kB$ in a neighbourhood of $\vec{0} \in \widehat{A}$. According to \cite[Proposition 3.2]{BottChern}, the connection  matrix $P$ of $\nabla$ relative to the  frame $F$ is given by the formula
$
P = \partial(C) C^{-1},
$
where $C$ is the Gram matrix of $F$. By Theorem \ref{T:GramMatrixKVWSpace}, $C$ is a scalar matrix. As a consequence, $P$ is a scalar matrix as well. By \cite[Corollary 3.10]{BottChern}, the curvature matrix of $\nabla$ in the frame $F$ is equal to $\bar\partial(P)$, which is again scalar. It implies that the canonical connection $\nabla$ is projectively flat, as asserted.  
\end{proof}

\begin{example} For Wen's matrix $K$ given by (\ref{E:WenFroehlich})  the first Chern class of $\kB$ is given by the formula
\begin{equation}\label{E:ChernClassMagneticSpecial}
c_1(\kB) = -\frac{\bigl(p(g-1) +1\bigr)}{2t}i \sum\limits_{r = 1}^g  d\xi_{r} \wedge d\bar{\xi}_r + \frac{p}{t}i \sum\limits_{1 \le r < s \le g}  d\xi_{r} \wedge d\bar{\xi}_s.
\end{equation}
This is  a straightforward consequence of  (\ref{E:ChernClassMagnetic}) and the explicit formula (\ref{E:Kadjunct}) for the adjunct matrix of $K$. 
\end{example}

\begin{remark}\label{R:ProjFlat} Note that Theorem \ref{T:FMTCenterMass} applies to any symmetric and positive definite matrix $K \in \Mat_{g \times g}(\ZZ)$
(which \emph{need not be a Wen's matrix}) and gives  a vector bundle $\kC$ of rank $\delta$ on the dual torus $\widehat{A}$. For any $\vec\xi \in \CC^g$ we have a natural identification $\kC\Big|_{[\vec\xi]} \cong \mathsf{W}_{K, \vec\xi}$.
In fact, $\kC$ is isomorphic (as a holomorphic vector bundle) to the magnetic  vector bundle $\kB$; see Theorem \ref{T:MagneticBundle}. However, $\kC$ carries a \emph{different} hermitian metric $h_c$, which is defined  by the scalar product (\ref{E:scalarproductmv}). Theorem \ref{T:HeisenbergGroupRepr}
 implies that the Bott--Chern connection of the hermitian holomorphic vector bundle $(\kC, h_c)$ is projectively flat. 
\end{remark}

\section{Restricted magnetic vector bundle}\label{S:RMultilayer}
Let $\bigl(E, (K, \vec\nn)\bigr)$ be an arbitrary KV-W datum, $\kB \cong \kC$ be the corresponding magnetic bundle  on $\widehat{A} =  \underbrace{\widehat{E} \times \dots \times \widehat{E}}_{g \; \mbox{\scriptsize{\sl times}}}$ (see Theorem \ref{T:MagneticBundle} and Theorem \ref{T:FMTCenterMass}). 

\begin{theorem}\label{T:MagneticBundleCurve}
Let $\widehat{E} \stackrel{\imath}\lar \underbrace{\widehat{E} \times \dots \times \widehat{E}}_{g \; \mbox{\scriptsize{\sl times}}}$ be the diagonal embedding and $\kU := \kB\big|_{\widehat{E}}$ be the \emph{restricted magnetic bundle} of the multi-layer torus model of FQHE. Then $\kU$ has rank $\delta$ and degree $-\rho$. In particular, its slope is 
\begin{equation}\label{E:Slope}
\frac{\mathsf{deg}(\kU)}{\mathsf{rk}(\kU)}  = - \dfrac{\rho}{\delta}.
\end{equation}
Moreover, $\kU$ is stable if and only if the matrix $K$ is primary. 
\end{theorem}
\begin{proof}
We deduce from  (\ref{E:FirstChernClass}) and 
(\ref{E:SumAdjunct}) that
$$
c_1(\kU) = \imath^\ast\bigl(c_1(\kB)\bigr) = -\dfrac{i}{2t} \sum_{k, l = 1}^g 
K^\sharp_{kl}  d\xi \wedge d\bar{\xi} = -\dfrac{i \rho}{2t} d\xi \wedge d\bar{\xi},
$$
hence $\mathsf{deg}(\kU) = -\rho$. 
A vector bundle on an elliptic curve  is simple if and only if it is stable; 
see for instance \cite[Corollary 4.5]{BK3}. Moreover, the rank and degree of a simple vector bundle on an elliptic curve  are coprime; see \cite[Theorem 10]{Atiyah}.
Hence, if  $\kU$ is stable then  the matrix $K$ is primary. 

\smallskip
\noindent
Conversely, suppose that $K$ is primary. 
We have a morphism of abelian varieties  
$$\widetilde{E}:= \CC/(\ZZ + \delta^{-1}\tau \ZZ) \stackrel{\varepsilon}\lar
\widehat{B} = \CC^g/(\ZZ^g + \tau K^{-1} \ZZ^g), \quad 
[z] \mapsto \bigl[\delta z \vec{u}\bigr].
$$
Consider the \'etale covering $\widetilde{E} \stackrel{\gamma}\lar \widehat{E}, [z] \mapsto [\delta z]$. Then the diagram 
\begin{equation}\label{E:PullBackAbVar}
\begin{array}{c}
\xymatrix{
\widetilde{E} \ar[r]^-{\varepsilon} \ar[d]_-{\gamma} & \widehat{B} \ar[d]^-{\widehat\kappa} \\
\widehat{E} \ar[r]^-\imath &  \widehat{A}
}
\end{array}
\end{equation}
is commutative, where $\widehat{B} = \CC^g/(\ZZ^g + \tau K^{-1}\ZZ^g) \stackrel{\widehat{\kappa}}\lar \widehat{A} = \CC^g/(\ZZ^g + \tau \ZZ^g)$ is given by the formula $\widehat{\kappa}\bigl([\vec{w}]\bigr) = \bigl([K\vec{w}]\bigr)$. Moreover, (\ref{E:PullBackAbVar}) is a pull-back diagram in the category of complex algebraic varieties.  One can check the universal property of a pull-back in a straightforward way. Alternatively, let 
$\begin{array}{c}
\xymatrix{
\overline{E} \ar[r]^-{\psi} \ar[d]_-{\phi} & \widehat{B} \ar[d]^-{\widehat\kappa} \\
\widehat{E} \ar[r]^-\imath &  \widehat{A}
}
\end{array}
$
be a pull-back diagram in the category of complex algebraic schemes. By \cite[Proposition I.3.3]{Milne} $\overline{E}$ is an elliptic curve and $\phi$ is an \'etale morphism of degree $\delta$. By the universal property of the pull-back we get a uniquely determined morphism $\overline{E} \stackrel{\widetilde\phi}\lar \widetilde{E}$ such that the following diagram 
$$ 
\xymatrix{
\overline{E} \ar@/_/[ddr]_-\phi \ar@/^/[drr]^-\psi
\ar@{.>}[dr]|-{\widetilde\phi} \\
& \widetilde{E} \ar[d]_-\gamma \ar[r]^-\varepsilon
& \widehat{B} \ar[d]^-{\widehat\kappa} \\
& \widehat{E} \ar[r]^-\imath & \widehat{A}.}
$$
is commutative. Since $\phi$ and $\gamma$ are both \'etale morphisms of the same degree, $\widetilde\phi$ is an isomorphism. 

\smallskip
\noindent
Let $\kV := \FF_B(\kT)$. Then we have: $\kB \cong \widehat\kappa_\ast \kV$;  see Proposition \ref{P:FMTdualIsogeny}.
Now, we have a base-change isomorphism
$
\kU = \imath^\ast\bigl(\widehat{\kappa}_\ast(\kV)\bigr) \cong \gamma_\ast\bigl(\varepsilon^\ast(\kV)\bigr).
$ It follows from the Riemann-Roch formula that
$$
-\rho = \mathrm{deg}(\kU) = \chi(\kU) = 
\chi\bigl(\varepsilon^\ast(\kV)\bigr) = \mathsf{deg}\bigl(\varepsilon^\ast(\kV)\bigr).
$$
Since the degree of the isogeny $\gamma$ and the degree of the line bundle 
$\varepsilon^\ast(\kV)$ are coprime, it follows from  \cite[Theorem 1.2]{Oda}
that $\End_{\widehat{E}}(\kU) \cong \CC$. Hence, the vector bundle $\kU$ is stable; see for instance \cite[Corollary 4.5]{BK3}.
\end{proof}

\begin{remark}
One can explicitly describe the vector bundle $\kU$ in terms of automorphy factors; see for instance \cite[Proposition 4.1.6]{BK4}.
\end{remark}

\begin{remark}
The minus sign in the  expression of the degree  of the vector bundle $\kU$ might look irritating. Note, however, that the same sign appears in a work of Kohmoto identifying the first Chern class of the  magnetic line bundle with the Hall conductance in the integral quantum Hall effect; see \cite[formula (4.9)]{Kohmoto}.
\end{remark}

\begin{remark}\label{R:Jain} Let $K = K_{p, g}$ be the matrix given by (\ref{E:WenFroehlich}). Then  the absolute value of the slope of $\kU$  is a Jain fraction (see \cite{Jain}): 
\begin{equation}\label{E:JainFraction}
\left|\frac{\mathsf{deg}(\kU)}{\mathsf{rk}(\kU)} \right| = \dfrac{g}{gp +1}.
\end{equation}
\end{remark}

\smallskip
\noindent
\emph{Acknowledgement}. We are grateful to M.~Zirnbauer and D.~Zvonkine for many enlightening discussions of the results of this paper as well as to both anonymous referees for
their helpful comments and remarks. Our work was partially  supported by the  German Research Foundation SFB-TRR 358/1 2023 — 491392403, by the Initiative
d’excellence (Idex) program and the Institute for Advanced Study Fellowship of the University
of Strasbourg, and the ANR-20-CE40-0017 grant.  

\smallskip
\noindent
\emph{Declarations}. 

\smallskip
\noindent
\emph{Conflict of interest}.  The authors have no competing interests to declare that are relevant to the content of this
article.

\end{document}